\documentclass[11pt,article,reqno]{amsart}
\usepackage{uva_template}
\usepackage{todonotes}
\usepackage{comment}
\usepackage{dirtytalk}
\def\setminus{-} 
\def\trep{\tr_{F/\Q}(\ep)}
\definecolor{aquamarine}{rgb}{0.5, 1.0, 0.83}

\subjclass[2020]{}

\thanks{}
\title{Class Number Formulas for Certain Biquadratic Fields}
\begin{document}

% author information
% first author 
\author{Elizabeth Athaide}
\address{Department of Mathematics, Massachusetts Institute of Technology, 77 Massachusetts Avenue, 
Cambridge, MA 02139}
\email{eathaide@mit.edu}

% second author
\author{Emma Cardwell}
\address[]{Department of Mathematics, Harvard University, 1 Oxford Street, Cambridge, MA 02138}
\email[]{ecardwell@college.harvard.edu}

\author{Christina Thompson}
\address[]{Department of Mathematics, Stanford University, 450 Jane Stanford Way, Stanford, CA 94305}
\email[]{cthomps@stanford.edu}

%\date{\today}
\maketitle

\color{black}

\black{\begin{abstract} 
    We consider the class numbers of imaginary quadratic extensions $F(\sqrt{-p})$, for certain primes $p$, of totally real quadratic fields $F$ which have class number one. Using seminal work of Shintani, we obtain two elementary class number formulas for many such fields. The first expresses the class number as an alternating sum of terms that we generate from the coefficients of the power series expansions of two simple rational functions that depend on the arithmetic of $F$ and $p$. The second makes use of expansions of $1/p$, where $p$ is a prime such that $p \equiv 3 \pmod{4}$ and $p$ remains inert in $F$. More precisely, for a generator $\ep_F$ of the totally positive unit group of $\cao_F$, the base-$\ep_{F}$ expansion of $1/p$ has period length $\ell_{F,p}$, and our second class number formula expresses the class number as a finite sum over disjoint cosets of size $\ell_{F,p}$.
\end{abstract}
\section{Introduction}

 \black{The theory of class numbers has a rich history, beginning with Gauss's effort to understand how primes could be represented by positive definite binary quadratic forms \cite{cox}. Gauss recognized that $\SL_2(\Z)$ acts naturally on positive definite integral binary quadratic forms $f(X, Y) = aX^{2} + bXY + cY^{2}$}  \black{with fixed discriminant $-d = b^{2} - 4ac$. He proved that the set of equivalence classes under this action is a finite abelian group; the order of this group is known as the \emph{class number} $h(-d)$.} \black{The class group for quadratic forms of discriminant $d$ is also isomorphic to the ideal class group for the ring of integers of the quadratic field $\Q(\sqrt{d})$. Therefore, it is natural to ask whether results about Gauss's class numbers are glimpses of results for the class numbers $h_K$ of more general number fields $K$. In this spirit, we recall two surprising results for Gauss's class numbers.}
 
\black{\textcolor{black}{In the 1970s, Hirzebruch \cite{hirzebruch_sqrt_p} and Zagier \cite{zagier_sqrt_p} found an elegant 
%class number 
formula for $h(-p)$}, when $7\leq p\equiv 3 \pmod{4}$ is prime and $h(4p) = 1$. \textcolor{black}{If t}he simple continued fraction for $\sqrt{p}$ \black{is written as}
%has the following form
\begin{align*}
    \sqrt{p} = a_{0} + \dfrac{1}{a_{1} + \dfrac{1}{a_{2} + \dfrac{1}{\ddots}}} = [a_0,\overline{a_1,\ldots, a_{2t}}], 
\end{align*}
where the repeating period begins with $a_1$ and has minimal even length $2t$, they proved that 
\begin{align}
\label{hirzebruch_zagier_formula}
    h(-p) = \dfrac{1}{3} \sum_{k=1}^{2t}(-1)^{k}a_{k}.
\end{align} 
% before:  elegant and surprising
More recently in the 1990s, Girstmair \cite{girstmair_g_expansions} found another elegant formula as an alternating sum of numbers that are even simpler to describe. Namely, if $g$ is a primitive root modulo $p$, he examines the base $g$ expansion of $1/p$, which is eventually periodic with period length $p-1$ (see \cite{hardy}, Section $9.6$). If this period is $\overline{x_{1}x_{2}...x_{p-1}}$, where $0 \leq x_i\leq g-1$, then he proved that
\begin{align}
\label{GirstmairFormula}
    h(-p) = \dfrac{1}{g+1}\sum_{k = 1}^{p -1} (-1)^{k}x_{k}. 
\end{align}}
% \ken{Alternating sum}

%class numbers are well known to be the class numbers of the rings of integers of quadratic number fields. 
\black{
A priori, these results are unexpected relationships between combinatorial sums and class numbers of binary quadratic forms with discriminant $-p$. \black{Since class numbers of binary quadratic forms are also examples of class numbers of number fields,} it is natural to ask whether \eqref{thm:1.1} and \eqref{thm 1.2} are glimpses of a more general theory where class numbers of number fields can be described as alternating sums of combinatorial numbers. We show that this is indeed the case for a large class of imaginary quadratic extensions of real quadratic fields $F$.}} \black{\black{To make this precise, suppose that $F = \Q(\sqrt{d})$, where $d > 1$ is square-free. Throughout, we assume that its ring of integers $\cao_F$ has class number 1. We note that $\cao_F = \Z[\theta_F]$, where we let
\begin{align*}
   \theta_F \coloneqq \begin{cases}
       \sqrt{d} &\quad \textrm{if }d\equiv 3\pmod{4}\\
        \frac{1 + \sqrt{d}}{2} &\quad \textrm{if }d\equiv 1\pmod{4}.
    \end{cases}
\end{align*} The imaginary quadratic extensions of $F$ that we consider are of the form $F(\sqrt{-p})$, where $p$ is a prime for which $7\leq p \equiv 3 \pmod{4}$ and $\legendre{d}{p} = -1$. These conditions imply that the relative discriminant ideal is the prime ideal $p\cao_F$ (see Lemma \ref{pof}). Moreover, for convenience, we fix a generator $\rho_{F,p} \coloneqq a + b \theta_F \in \cao_F$ such that $(\cao_F/p\cao_F)^{\times} = \langle \rho_{F,p} +p\cao_F\rangle\cong \F_{p^2}^{\times}$.} 

\black{In this setting, we derive a class number formula for $F(\sqrt{-p})$ as an alternating sum that arises from $p$ and invariants of $F$. Our key observation is that the combinatorial structure that underlies \eqref{hirzebruch_zagier_formula} and \eqref{GirstmairFormula} can be reformulated in terms of recurrence relations that can be captured by the coefficients of distinguished rational functions. Therefore, our goal is to define two rational functions (reflecting that $F$ \textcolor{black}{has degree 2 over $\Q$}) whose coefficients can be incorporated into an alternating sum that yields the class number $h_{F(\sqrt{-p})}$.

To this end, we use $\rho_{F,p} = a+b\theta_F$ to define integers 
    \begin{align}
    \label{def:c_defns}
        C_{F,p} &\coloneqq a^{2} + ab\cdot \tr_{F/\Q}(\theta_F) + \norm_{F/\Q}(\theta_F)b^{2}, \\ \label{def:d_defns}
        D_{F,p}  &\coloneqq 2a + b\cdot \tr_{F/\Q}({\theta_F}),
    \end{align}
and in turn, to define the rational functions as
    \begin{align}
\label{eqn:x1defns}
        X_{F,p}(z) &= \sum_{m \geq 1} x(m) z^m  \coloneqq  \dfrac{az -C_{F,p}z^2}{C_{F,p}z^2 - D_{F,p}z + 1}, \\ \label{eqn:x2defns}
        Y_{F,p}(z) &= \sum_{m\geq 1} y(m) z^m \coloneqq    \dfrac{bz}{C_{F,p}z^{2} - D_{F,p}z + 1}.
    \end{align}
Moreover, we must delicately take into account the presence of nontrivial units as they inform class number calculations. To make this precise, we recall that Dirichlet's Unit Theorem implies that $\cao_F^\times = \{\pm \ep_F^j, j\in \Z\}$, where $\ep_F = s+t\theta_F$ is the totally positive fundamental unit. We then define $t$ pairs of sequences, say  
$\left\{({x}_i(m), {y}_i(m))\;:\;  \;m\geq 1\right\}$, where $t$ is the coefficient of $\theta_F$ in $\ep_F$, that encode the action of $\ep_F$ by means of expressions involving $x(m)$ and $y(m)$ (see \eqref{tilde_defns}). Finally, we find that the analogues of the right hand side of \eqref{GirstmairFormula} turn out to be obtained from the quadratic form 
    \begin{align*}
     Q_{F}(Y_1, Y_2)\coloneqq \tr_{F/\Q}(\ep_F)Y_1^2+4Y_1Y_2+\tr_{F/\Q}(\ep_F)Y_2^2. 
    \end{align*}
In terms of this data, we obtain the following theorem, which \textcolor{black}{gives a formula for the class number} $h_{F(\sqrt{-p})}$.}}
\black{\begin{theorem}
\label{thm:1.1}
 Assuming the notation and hypotheses above, we have 
 \begin{align*}
     h_{F(\sqrt{-p})}=\frac{1}{16t^2 p^2}\sum_{\substack{1\leq m\leq p^2-1\\ 1\leq i\leq t}}(-1)^mQ_{F}\left({x}_i(m),{y}_i(m)\right).
 \end{align*}
\end{theorem}
\begin{remark}
For real quadratic fields $F$ with $h_F=1$, Theorem \ref{thm:1.1} applies for one-fourth of the primes. This follows from the strong version of Dirichlet's Theorem on primes in arithmetic progressions, which implies that the primes $p$ such that $p \equiv 3 \pmod 4$ and $\legendre{d}{p} = -1$ have density 1/4.\end{remark}}
\black{
\begin{example}
Here we illustrate Theorem ~\ref{thm:1.1} with $F=\Q(\sqrt{3})$ and $p = 7$. The field $F(\sqrt{-7})$ has class number $h_{F(\sqrt{-7})} = 2$. Note that $F$ has class number $1$, and its totally positive fundamental unit is $\ep_F= 2+\sqrt{3}$, and so $t=1$. The prime $p = 7$ satisfies the required conditions that $p\equiv 3\pmod 4$ and $\legendre{3}{7} = -1$. Therefore, we have that the principal ideal $7\cao_F\subset\O_F= \Z[\sqrt{3}]$ is prime, and so we have that $\cao_F/7\cao_F\cong \F_{49}$. One can check that $\rho_{F,p} = 6+\sqrt{3}$ generates the multiplicative cyclic group $(\O_F/7\O_F)^\times\cong \F_{49}^\times$. Thus we have $a=6, b=1,$ and using \eqref{def:c_defns} and \eqref{def:d_defns}, we find that $C_{F,p}=33,$ and $D_{F,p}=-12,$ which in turn by \eqref{eqn:x1defns} and \eqref{eqn:x2defns} give
        \begin{align*}
           X_{F,p}(z)  &= \sum_{m\ge1} x(m)z^m = 6z + 39z^2 + 270 z^3 + 1953z^4 + \ldots = \dfrac{6z - 33z^2}{33z^2 +12z+ 1} ,\\
           Y_{F,p}(z)  &= \sum_{m\ge1} y(m)z^m  = z+12z^2+111z^3 + 936z^4 + \ldots=  \dfrac{z}{33z^{2} +12z + 1} .
        \end{align*}
Theorem ~\ref{thm:1.1} offers a formula for $h_{F(\sqrt{-7})}$ as an alternating sum of $7^2-1=48$ terms that are assembled from the first $48$ coefficients of $X_{F,p}(z)$ and $Y_{F,p}(z)$. Furthermore, because $t=1$, the relevant pairs $\{x_1(m), y_1(m)\}$ are merely reductions of the pairs of coefficients $\{x(m),y(m)\}$ to a specific \emph{fundamental domain}, as given in \eqref{tilde_defns}. One finds that 
\begin{table}[h]
    \centering
    \begin{tabular}{l l l l}
    $  {x}_{1}(1) = 1,$   & ${x}_{1}(2) = -5,$ & $\ldots,$ & $  {x}_{1}(48) = -5,$  \\
    $y_1(1) = -5,$
         & ${y}_{1}(2) = -4$, & $\ldots,$&  ${y}_{1}(48) = -7.$\
    \end{tabular}
\end{table}

\noindent \textcolor{black}{We now use Theorem ~\ref{thm:1.1} to calculate $h_{F(\sqrt{-7})}$:}
\begin{align*}
    h_{F(\sqrt{-7})}&=\frac{1}{784}\sum_{1\leq m\leq 48}(-1)^m\bigg[4x_1(m)^2+4x_1(m)y_1(m)+4y_1(m)^2\bigg]\\
    &= \frac{1}{784}(-84+76-300+52-28+\cdots + 436)=2.
\end{align*}
\end{example}}

\black{We circle back to the fact that the class number formula in \eqref{GirstmairFormula} makes use of the base $g$ expansion of $1/p$. We stress that the number of terms in the sum, which is $p-1$, is the length of the repeating period of this expansion. Therefore, we ask whether the expression in Theorem ~\ref{thm:1.1} can be reformulated so that the number of terms in the sum equals the period length of an analogous expansion of $1/p$. We find, indeed, that this is the case.}

\black{In the setting of Theorem ~\ref{thm:1.1}, it is natural to consider the \emph{base-$\ep_{F}$ expansion} of elements $
\alpha \in F$. To be precise, there is a unique sequence of integers $a_n, a_{n-1}, \ldots, a_{0}, a_{-1}, a_{-2}, \ldots$, with $0 \le a_{i} \le \lfloor \ep_{F} \rfloor$, for which 
\begin{align}
    \alpha = a_{n}\ep_F^n + a_{n-1}\ep_{F}^{n-1} + \ldots + a_0 + a_{-1}\ep_{F}^{-1} + a_{-2}\ep_{F}^{-2} + \ldots .
\end{align}
The above expression is called the \emph{base}-$\ep_{F}$ \emph{expansion} of $\alpha$, and it is well-known that such expansions are eventually periodic (see, for example, \cite{pisot}). To recast Theorem ~\ref{thm:1.1} in terms of these expansions, we require the following finite set:
\begin{align}
    R_{F,p} := \left\{r_1+r_2\ep_F\in\frac{1}{p}\O_F : r_1\in\Q\cap(0,1],r_2\in\Q\cap[0,1)\right\},
\end{align}
which is known as the \emph{Shintani set} for $F$ at $p$, when $p\equiv 3\pmod 4$ and $\legendre{d}{p}=-1$. The totally positive units define a group action of $\O_F^{\times,+}\coloneqq\langle\ep_F\rangle$ onto $R_{F,p}$ as follows.
\[
    \ep_F \ast (r_1+r_2\ep_F) \coloneqq (1-r_2) + \{r_1+r_2 \tr_{F/\Q}(\ep_F)\}\ep_F, 
\]
where $\{x\}\coloneqq x - \lfloor x\rfloor$ is the fractional part of $x$. Under this action, the set $R_{F,p}$ is a finite disjoint union of orbits, say $$R_{F,p} = \bigsqcup_{r\in \cao_F^{\times,+}\backslash R_{F,p}} \cao_F^{\times, +}\ast r.$$ For $r\in R_{F,p} \backslash \cao_F$, we prove (see Lemma ~\ref{orbit_length_is_cycle_length}) that the number of elements in the orbit of $r$ under $\ep_F$ is equal to the period length of $1/p$ in base $\ep_F,$ which we denote $\ell_{F,p}.$ This allows us to now state the desired class number formula as a sum over $\ell_{F,p}$ terms, where we make the following abuse of notation:
    \begin{align*}
         {Q}_F(r_1 + r_2\ep_{F}) =  Q_{F}(r_1, r_2).
    \end{align*}
    \begin{theorem}\label{thm 1.2}
    Assuming the notation and hypotheses from Theorem ~\ref{thm:1.1}, we have
    \[
        h_{F(\sqrt{-p})}=\frac{1}{4}\sum_{i=1}^{\ell_{F,p}}\hspace{+0.23cm}\sum_{r\in \cao_F^{\times,+}\backslash R_{F,p}} \chi_{F(\sqrt{-p})/F} \left({rp}\cao_{F}\right) Q_F(\ep_{F}^i \ast r),
    \]
    where $\chi_{F(\sqrt{-p})/F}$ is the unique quadratic Hecke character of conductor $p\cao_{F}$. 
\end{theorem}}
\black{
    \begin{example}
        Now we illustrate Theorem ~\ref{thm 1.2} with $F=\Q(\sqrt{3})$ and $p=7$, where $h_F=1$ and $\ep_F=2+\sqrt{3}$ (so $t=1$). One can check (for example, using {\tt{SageMath}}) that the base-$\ep_F$ expansion of $1/7$ is \begin{align*}
            \frac{1}{7}&=\ep_F^{-2} +\sum_{i=0}^\infty \left({3\ep_F^{-8i-3}+2\ep_F^{-8i-4}+2\ep_F^{-8i-5}+2\ep_F^{-8i-7}+2\ep_F^{-8i-8}+3\ep_F^{-8i-9}}\right)\\
            &=0.01\overline{32202230}.
        \end{align*}
        Thus, we see that the base $\ep_F$ expansion of $1/7$ has period length $\ell_{F,7}=8$. Since $|R_{F,7}\setminus\cao_F|=tp^2-t=48$ (see Lemmas ~\ref{shin_set_construction} and ~\ref{explicit_kernel}), we deduce that $\cao_F^{\times,+}\backslash(R_{F,7}\setminus\cao_F)$ contains $48/\ell_{F,7}=6$ disjoint orbits. One can also verify that the set
        \begin{align*}
           \left\{\frac{1}{7}+\frac{1}{7}\ep_F,\hspace{+0.1 cm}\frac{1}{7},\hspace{+0.1 cm}\frac{1}{7}+\frac{4}{7}\ep_F,\hspace{+0.1 cm}\frac{1}{7}+\frac{5}{7}\ep_F,\hspace{+0.1 cm}\frac{2}{7}+\frac{2}{7}\ep_F,\hspace{+0.1 cm}\frac{3}{7}\right\}.
        \end{align*}
        is a complete set of orbit representatives for $\cao_{F}^{\times,+}\backslash (R_{F,7}\setminus \cao_F)$. Equipped with these values, Theorem ~\nolinebreak ~\ref{thm 1.2} states that
        \begin{align*}
            h_{F(\sqrt{-7})}&=\frac{1}{4}\sum_{i=1}^{8}\hspace{+0.23cm}\sum_{r\in R_{F,p}\backslash \cao_F^{\times,+}} \chi_{F(\sqrt{-7})/F} \left(r {p}\cao_{F}\right) Q_F(\ep_{F}^i \ast r)\\
            &=\frac{1}{4}\bigg(-\frac{220}{7} + \frac{228}{7} - \frac{188}{7} + \frac{212}{7} - \frac{180}{7} + \frac{204}{7}\bigg)=2.
        \end{align*}
    \end{example}}
    
\medskip
\black{Theorems ~\ref{thm:1.1} and ~\ref{thm 1.2} are generalizations of the results from Hirzebruch-Zagier and Girstmair to the setting of imaginary quadratic extensions of real quadratic fields $F$ with $h_F=1$.
Within this new setting, we prove our theorems by working with a class number formula analogous to the one used in the quadratic setting by Hirzebruch, Zagier, and Girstmair.

Both \eqref{hirzebruch_zagier_formula} and \eqref{GirstmairFormula} arise from a finite version of Dirichlet's class number formula, which relates the Dirichlet $L$-function, an infinite series, to the class number of $\Q(\sqrt{-d})$:
\begin{align*}
    L(1, \chi_{d}) = \dfrac{2\pi}{\omega\sqrt{d}}h(-d),
\end{align*}
where $\omega$ represents the number of roots of unity in $\Q(\sqrt{-d})$, and $\chi_d$ is a primitive Dirichlet character of conductor $d$. Using the functional equation of this $L$-function, the above equation can be written in terms of $L(0, \chi_{d})$, which in turn allows us to use the Hurwitz $\zeta$-function and the periodicity of $\chi_d$ to rewrite this class number formula as a finite sum of Bernoulli polynomials evaluated at integer points. Our work uses an analogous formula of Shintani \cite{shintani_thm}, which expresses the class numbers of totally imaginary quadratic extensions of totally real fields as finite sums assembled from Bernoulli numbers.}

\black{In Section ~\ref{sec2}, we review the background needed to state Shintani's class number formula for imaginary quadratic extensions of real quadratic fields $F$ with $h_F=1$. These formulae involve \say{Shintani sets,} which are something like fundamental domains for the action of the totally positive units on $\frac{1}{p}\cao_F$. The crux of our work relies on combinatorial properties of these sets, which we derive in Section ~\ref{sec2}.
Then, we prove Theorem ~\ref{thm:1.1} in Section ~\ref{section_1.1_proof} and Theorem ~\ref{thm 1.2} in Section ~\ref{section4}. Finally, in Section ~\ref{example_section}, we use Theorems ~\ref{thm:1.1} and ~\ref{thm 1.2} to calculate class numbers of $\Q(\sqrt{3},\sqrt{-p})$, where $p\equiv 3 \pmod{4}$ is prime, $\legendre{d}{p} = -1$, and $p<100$.}
\black{
\section*{Acknowledgements}\label{sec: intro}
The authors were participants in the 2023 UVA REU in Number Theory. They are grateful for the support of grants from Jane Street Capital, the National Science Foundation (DMS-2002265 and DMS- 2147273), the National Security Agency (H98230-23-1-0016), and the Templeton World Charity Foundation. The authors thank Ken Ono, Wei-Lun Tsai, Alejandro De Las Penas Castano, and Eleanor McSpirit for suggesting the problem and for their mentorship and support. \black{They would also like to thank Marie-H\'el\`ene Tom\'e and the other participants of the 2023 UVA REU for many thoughtful discussions.}}
\black{
\section{Shintani's Class Number Formula and Properties of Shintani Sets}\label{sec2}}
\black{In this section, we discuss the background needed to state Shintani's class number formula. While Shintani's theorem is true for totally imaginary quadratic extensions of a totally real field of arbitrary degree, we restrict the following commentary and definitions to the case that $F$ is quadratic with class number $1.$ Throughout this section, we fix a real quadratic field $F$ of class number is 1 and a totally imaginary quadratic extension of $F$, which we denote $K = F(\sqrt{-p})$, where $p \equiv 3\pmod{4}$.
\newline
\subsection{Algebraic Background} \label{alg-background}
Shintani's formula can be used to calculate the relative class number $h_K/h_F$ in terms of invariants of $F$, $K$, and the extension $K/F$ itself. Before stating the formula, we review the definitions of these invariants. }

\black{
The \emph{regulator} $R_L$  of a number field $L$ measures the density of units in the ring of integers. The regulator can be determined by considering the matrix $$[N_j\log(\sigma_j(u_i))],$$ where each $u_i$ is a fundamental unit from the set $u_1, \cdots, u_{k}$ generating the unit group in $\cao_L$, each $\sigma_j$ is a unique Archimedian place of $L$, and $N_j$ is defined to be $1$ if $\sigma_j$ is real, and $2$ if $\sigma_j$ is complex. If we define $r_1$ and $r_2$ respectively to be the number of real and complex embeddings of $L$, by Dirichlet's unit theorem, we see that this matrix has dimension $(r_1+r_2-1)\times (r_1+r_2)$. The regulator $R_L$ is the determinant of the square submatrix which is formed by deleting any single column of this matrix. Since the sum of the entries in each row of this matrix is $0$, this determinant is independent of which column is deleted. If we consider the rows of this matrix as forming a lattice in $\R^{r_1+r_2-1},$ then the regulator is directly proportional to the volume of the fundamental domain associated to this lattice. 

Next, we examine the unit groups of $\cao_K$ and $\cao_F.$ Since $F$ is a real quadratic field, and $K$ is a totally imaginary quadratic extension of $F$,  Dirichlet's unit theorem implies that $\cao_F$ and $\cao_K$ are both $\Z$-modules of rank $1$. More precisely, if we let $\mu_F$ and $\mu_K$ represent the groups of roots of unity in $F$ and $K$ respectively, there exists $\ep_F\in\cao_F$ and $\ep_K\in\cao_K$ such that $\cao_F^\times=\mu_F\times\langle\ep_F\rangle$ and $\cao_K^\times=\mu_K\times\langle\ep_K\rangle.$ Since $F$ is real quadratic and any of $\ep_F,-\ep_F,\ep_F^{-1},-\ep_F^{-1}$ can generate the free part of $\cao_F,$ we can choose $\ep_F$ to be totally positive and greater than $1.$ 
\black{
\begin{lemma}\label{ep_F_is_ep_K}
    We have $\cao_K^\times=\cao_F^\times.$ In particular, we may choose $\ep_K=\ep_F.$
\end{lemma}
\begin{proof}
A theorem by Fr\"olich and Taylor shows that $[\cao_K^\times:\cao_F^\times\mu_K]=1\text{ or }2$ (see Theorem 42 in  \cite{taylor}). Since $K=\Q(\sqrt{d},\sqrt{-p})$ for $p\geq 7,$ $\mu_K=\{\pm1\}.$ Thus $\mu_K=\mu_F,$ so $[\cao_K^\times:\cao_F^\times]=1\text{ or }2.$

\black{Now, assume for the sake of contradiction that $[\cao_K^\times:\cao_F^\times]=2,$ so $\ep_K \notin \cao_F^{\times},$ and $\ep_K^2\in\cao_F^\times$. Since $K\cap\R=F,$ we see that $\ep_K\not\in\R.$ However, we know that $\ep_K^2 \in \cao_F \subset \R$. Observe that both $\ep_K \in \C \setminus \R$ and $\ep_K^2 \in \R$ if and only if $\mathrm{Re}(\ep_K) = 0$. Additionally, $\norm_{K/\Q}(\ep_K) = \pm 1$, which implies that $\ep_K = \pm i.$ However, this is a contradiction since $\pm i \notin K,$ so we see that $[\cao_K^\times:\cao_F^\times]=1$, and hence $\cao_K^\times=\mu_F\times\langle\ep_F\rangle=\mu_K\times\langle\ep_F\rangle,$ so we can choose $\ep_K = \ep_F$.}
\end{proof}}
Equipped with the fact that $\ep_K = \ep_F,$ we may now relate the regulators $R_K$ and $R_F$ of $K$ and $F,$ which we do in the lemma which follows.

\begin{lemma}\label{regulator_lemma}
    For $F$ and $K$ as defined in the beginning of this section, we have that $R_K=2R_F.$
\end{lemma}
\begin{proof}
Since $\ep_F=\ep_K$ by Lemma ~~\ref{ep_F_is_ep_K}, the regulators of the fields $F$ and $K$ as previously defined are determined using the following matrices:
\begin{align*}
    R_F: \begin{bmatrix}
        \log|\ep_F| && \log|-\ep_F|
    \end{bmatrix}\\
    R_K: \begin{bmatrix}
        2\log|\ep_F| && 2\log|-\ep_F|
    \end{bmatrix}
\end{align*}
Thus $R_F=\log|\ep_F|$ and $R_K=2\log|\ep_F|=2R_F.$
\end{proof}

\black{Next, we review the definition of the \emph{relative discriminant ideal} $D_{K/F}$ for our fields $K$ and $F$. Recall that $F$ has class number $1$, so $D_{K/F}$ is principal. Since $K/F$ is quadratic, it is Galois, and its Galois group consists of two elements: the identity and complex conjugation. In this setting, $D_{K/F}$ is given by}
\begin{align*}
     D_{K/F}\coloneqq \left(\det \begin{bmatrix}
            \omega_1 && \omega_2\\
            \overline{\omega_1} && \overline{\omega_2}
        \end{bmatrix}\right)^2\cao_F,
\end{align*}
where $\{\omega_1,\omega_2\}$ is an integral basis of $K/F$. We know that an integral basis will exist in our case by the following argument. From the structure theorem for finitely generated modules over a Dedekind domain, we have that $\cao_K\cong\cao_F^n\oplus\mathfrak{a},$ where $\mathfrak{a}$ is an ideal of $\cao_F$ and $n\in\Z_{\geq0}$ (see Theorem $1.32$ of \cite{narkiewicz}). Since $h_F=1,$ implying $\cao_F$ is a principal ideal domain, $\cao_K$ must be a free $\cao_F$-module of rank $2=[K:F].$
\black{\begin{lemma}\label{pof}
    The set $\{1, \frac{1+\sqrt{-p}}{2}\}$ is an integral basis of $K/F,$ and thus we have $D_{K/F}=p\cao_F.$
\end{lemma}
\begin{proof}
     Let $A$ be the change-of-basis matrix from the integral basis $\{\omega_1,\omega_2\}$ to the $F$-basis $\{1, \frac{1+\sqrt{-p}}{2}\}$. We see that
\begin{align*}
    p\cao_F = \left(\det \begin{bmatrix}
            1 && \frac{1+\sqrt{-p}}{2}\\
            1 && \frac{1-\sqrt{-p}}{2}
        \end{bmatrix}\right)^2\cao_F = (\det A)^2 D_{K/F}.
\end{align*}
Since $\cao_F$ is a Dedekind domain, ideals in $\cao_F$ factor uniquely. Therefore since $p\cao_F$ is prime by assumption, $\det A$ must be a unit in $\cao_F$, so $A \in \mathrm{GL}_2(\cao_F)$. Thus, $\{1, \frac{1+\sqrt{-p}}{2}\}$ is an integral basis of $K/F$. Using this integral basis, we see that
    \begin{align*}
D_{K/F}=\left(\frac{1-\sqrt{-p}}{2}-\frac{1+\sqrt{-p}}{2}\right)^2\cao_F=p\cao_F.
    \end{align*}
\end{proof}}

Finally, since $\gal(K/F)\cong \Z/2\Z$, there is a unique nontrivial character $\chi:\gal(K/F)\to\C^\times$. By class field theory, we can consider the precomposition of $\chi$ with the Artin symbol to obtain a character $\chi_{K/F}$ of the group of fractional ideals that are relatively prime to $D_{K/F}$. This is known as the \emph{Hecke character of $K/F$ with conductor $D_{K/F}$.} By definition of the Artin symbol (see, for example, \cite{cox} page 106), we can explicitly compute the value of $\chi_{K/F}$ for any prime ideal $\mathfrak{p}$:
    \begin{align*}
        \chi_{K/F}(\mathfrak{p})= \begin{cases} 
        1 & \mathfrak{p} \text{ splits in } \cao_{K}\\
        -1 & \mathfrak{p} \text{ remains inert in } \cao_{K}\\
        0 & \mathfrak{p} \text{ ramifies in } \cao_{K}.
        \end{cases}
    \end{align*}
\begin{rmk}
    Shintani's class number formula relies on the narrow ideal class group character with conductor $D_{K/F}$ evaluated at fractional ideals. This corresponds to a primitive Grössencharakter with modulus $D_{K/F}$ (see Prop. 6.9 in \cite{neukirch}). Since $\mathrm{Gal}(K/F)\cong \Z/2\Z$, the nontrivial character $\chi: \mathrm{Gal}(K/F)\to\C^\times$ is unique and injective, so class field theory implies that the primitive Grössencharacter with modulus $D_{K/F}$ is unique and corresponds to $\chi$. Hence, we can see that the character used in Shintani's formula is exactly the Grössencharakter. For more details, see Sections 6 and 10 in \cite{neukirch}. %\emma{CHECK SECTION NUMBERS}
\end{rmk}
\begin{rmk} \label{mult_ideal_by_unit}
   For any unit $u \in K$ and any ideal $\mathfrak{a}\subset\cao_{K}$, $u\cdot\mathfrak{a}=\mathfrak{a},$ and hence $\chi_{K/F}(u\cdot\mathfrak{a})=\chi_{K/F}(\mathfrak{a}).$
\end{rmk}

\medskip
\subsection{Shintani's Class Number Formula}
In this section, we prove a simplified version of Shintani's formula for real quadratic base fields $F$ with $h_F = 1$. 
\begin{prop}\label{shintani_simplified}
    For a totally real quadratic extension $F$ of $\Q$ with $h_{F} = 1$ and $K = F(\sqrt{-p})$ a totally imaginary quadratic extension of $F$ where $7\leq p \equiv 3\pmod{4}$ remains inert in $\cao_F$, Shintani's formula simplifies to the following:
\begin{align*}
    h_{K} = \frac{1}{2}\sum_{r\in R_{F,p}}{\chi_{K/F}\bigg((r_1+r_2\ep_F) D_{K/F}\bigg)\sum_{\substack{0\leq l_1,l_2\leq2\\ l_1+l_2=2}}{\frac{B_{l_1}(r_1)}{l_1!}\frac{B_{l_2}(r_2)}{l_2!}\mathrm{Tr}_{F/\Q}(\varepsilon_F)^{l_2-1}}},
\end{align*}
where $$R_{F,p} = \left\{r=r_1+r_2\ep_F\;:\; 0<r_1\leq 1, 0\leq r_2< 1, r\in \tfrac{1}{p}\cao_F\right\},$$
and $B_n(x)$ is the degree $n$ Bernoulli polynomial. As $ [F:\Q] = 2$, Shintani's formula only requires the following Bernoulli polynomials:
\begin{align*}
    B_0(x)=1,\quad
    B_1(x)=x-\frac{1}{2},\quad
    B_2(x)=x^2-x+\frac{1}{6}.
\end{align*}
\end{prop}
\begin{proof}
We follow \cite{shintani_thm} by first considering the embedding $F\to\R^2$ via 
\begin{align*}
    F\hookrightarrow \R^2 \quad \alpha \mapsto \left(\alpha, \alpha'\right),
\end{align*}
where $\alpha\mapsto\alpha'$ is the nontrivial automorphism in $\gal(F/\Q)$. Shintani shows that the first quadrant $\R^2_+\coloneqq\{(x,y)\in \R^2\;:\; x,y>0\}$ can be decomposed as the following disjoint union:
\begin{align*}
    \R_+^2 &= \bigcup_{\eta\in \cao_F^{\times,+}}\eta C_1\sqcup \bigcup_{\eta\in \cao_{F}^{\times,+}}\eta C_2
\end{align*} 
where $C_1$ is generated by the images of $1,\ep_F$ in $\R^2$ and $C_2$ is generated by the image of $1$:
\begin{align*}
    C_1= \{\lambda_1(1,1) + \lambda_2 (\ep_F, \ep_F') \in \R^2 \;:\; \lambda_1,\lambda_2>0\},\quad C_2 = \{\lambda (1,1) \in \R^2 \;:\; \lambda>0 \},
\end{align*}
and $\eta\in \cao_F^{\times,+}$ acts by component-wise multiplication. Next, for each cone $C_i$, Shintani defines the set $R(i,\frac{1}{p}\cao_F)$ \black{as the following vectors with components in $\Q\cap (0,1]$:}
\begin{align*}
    R\left(1,\tfrac{1}{p}\cao_F\right) &\coloneqq \left\{(r_1,r_2)\in \Q^2 \;:\; 0< r_1,r_2\leq 1, \; r_1+r_2\ep_F\in\tfrac{1}{p}\cao_F\right\}\\
    R\left(2, \tfrac{1}{p}\cao_F\right) &\coloneqq \left\{r_3 \in \Q \;:\; 0< r_3 \leq 1, \; r_3\in \tfrac{1}{p}\cao_F\right\}.
\end{align*}
Let $\chi_{K/F}$ be the unique quadratic character of the narrow ideal class group of $F$ with conductor $p\cao_F$, associated to $K$. Then, assuming the notation above, we have the class number formula 
 \begin{multline*}
         h_{K} = \frac{2\omega_K R_F}{R_K \left[\cao_F^\times:\cao_F^{\times,+}\right]} \left(\sum_{r\in R\left(1,\tfrac{1}{p}\cao_F\right)} \chi_{K/F} \bigg(\left(r_1 + r_2\ep_F\right)p\cao_F\bigg)\sum_{\substack{(l_1,l_2)\in \Z_{\geq0}^2\\ l_1+l_2 = 2}} \frac{B_{l_1}(r_1)B_{l_2}(r_2)}{2\cdot l_1!l_2!} \mathrm{Tr}_{F/\Q}\left(\ep_F^{l_2-1} \right)\right.\\
          - \left.\sum_{r_3\in R\left(2,\tfrac{1}{p}\cao_F\right)} \chi_{K/F}(r_3 p\cao_F)B_1(r_3)\right),
    \end{multline*}   
\noindent where $\omega_K$ is the number of roots of unity in $K$ (\hspace{1sp}\cite{shintani_thm}, Theorem 2). 

We first simplify the coefficient term in this formula. Recall from Lemma ~~\ref{regulator_lemma} that $R_F/R_K ~=~ 1/2$. Since $K = \Q(\sqrt{d},\sqrt{-p})$ for $p\geq 7,$ we have $\omega_K = 2.$ Furthermore, since we may choose the fundamental unit of $F$ to be totally positive, we see that $\cao_F^\times = \{\pm 1\} \times \cao_F^{\times,+}$, so we get $[\cao_F^\times:\cao_{F}^{\times,+}] = 2$. Then, 
\begin{align*}
    \frac{2 \omega_K R_F}{R_K \left[\cao_F^\times:\cao_F^{\times,+}\right]}  &= 1.
\end{align*}
Next, we reindex the sum. First we split the set $R(1, \frac{1}{p}\cao_F)$ into two parts. Consider the sets $R_1,R_2$ given by 
\begin{align*}
    R_1 &\coloneqq \left\{(r_1,r_2)\in \Q^2\;:\; 0<r_1\leq 1, 0< r_2< 1, r_1+r_2\ep_F\in\tfrac{1}{p}\cao_F\right\}\\
    R_2 &\coloneqq \left\{(r_1,r_2)\in \Q^2\;:\; 0<r_1\leq 1, r_2=1, r_1+r_2\ep_F\in\tfrac{1}{p}\cao_F\right\}.
\end{align*}
Additionally, for simplicity we denote the inner sum of Shintani's formula by: 
\begin{align*}
    \mathcal{B}(r_1+r_2\ep_F) \coloneqq \sum_{\substack{(l_1,l_2)\in \Z_{\geq0}^2\\ l_1+l_2 = 2}} \frac{B_{l_1}(r_1)B_{l_2}(r_2)}{2\cdot l_1!l_2!} \tr_{F/\Q}\left(\ep_F^{l_2-1} \right).
\end{align*}
By splitting the sum with $R(1,\tfrac{1}{p}\cao_F) = R_1\bigsqcup R_2$, we see that
\begin{align}\notag
   &h_{K} = \sum_{r\in R_1} \chi_{K/F} \bigg(\left(r_1 + r_2\ep_F\right)p\cao_F\bigg)\mathcal{B}(r_1+r_2\ep_F) \\\label{merging_Rs}
   &\hspace{2 cm} + \sum_{r\in R_2} \chi_{K/F} \bigg(\left(r_1 + \ep_F\right)p\cao_F\bigg)\mathcal{B}(r_1+\ep_F)- \sum_{r\in R\left(2,\tfrac{1}{p}\cao_F\right)} \chi_{K/F}(r p\cao_F)(r-1/2).
\end{align}
Since $\ep_F\in\cao_F^\times$ and $\chi_{K/F}$ has conductor $p\cao_F,$ we have $\chi_{K/F}\big(r_1p\cao_F\big)=\chi_{K/F}\Big((r_1 + \ep_F)p\cao_F\big)$. Moreover, comparing $\mathcal{B}(r_1)$ and $\mathcal{B}(r_1+\ep_F),$ we see that
\black{\begin{align*}
    \mathcal{B}(r_1)&=\frac{r_1^2 - r_1 + 1/3}{4}\tre - \frac{r_1-1/2}{2} \\
    \mathcal{B}(r_1+\ep_F)
    &=\frac{r_1^2 - r_1 + 1/3}{4}\tre + \frac{r_1-1/2}{2}=\mathcal{B}(r_1)+r_1-1/2.
\end{align*}}Thus \eqref{merging_Rs} simplifies to
\begin{align*}
    h_{K} = \frac{1}{2}\sum_{r\in R_{F,p}} \chi_{K/F} \bigg(\left(r_1 + r_2\ep_F\right)p\cao_F\bigg)\sum_{\substack{(l_1,l_2)\in \Z_{\geq0}^2\\ l_1+l_2 = 2}} \frac{B_{l_1}(r_1)B_{l_2}(r_2)}{ l_1!l_2!} \tr_{F/\Q}\left(\ep_F^{l_2-1} \right)
\end{align*}
where $R_{F,p}$ is given by
\begin{align*}
    R_{F,p} & = \left\{r=r_1+r_2\ep_F\;:\; 0<r_1\leq 1, 0\leq r_2< 1, r\in \tfrac{1}{p}\cao_F\right\}.
\end{align*}
\end{proof}
\black{
\begin{definition}
    We call $R_{F,p}$ the \textbf{Shintani set} associated to $F$ and $p$.
\end{definition}}
\black{
\subsection{Properties of Shintani Sets}
In this subsection, we identify a correspondence between $R_{F,p}$ and the finite field $\F_{p^2}$, which will play an important role in our proof of Theorem ~~\ref{thm:1.1}. Namely, we make use of this correspondence and the cyclic structure of the multiplicative group $\F_{p^2}^{\times}$ to enumerate the elements of $R_{F,p}\setminus\cao_F$ using the powers of a generator of $\F_{p^2}^{\times}$.}}

\black{Throughout this subsection, we fix a totally real quadratic field $F$ and an imaginary quadratic extension $K = F(\sqrt{-p})$, where $p \equiv 3\pmod{4}$ and $p$ remains inert in $\cao_F$. We let $\ep_F=s+t\theta_F,$ and to simplify notation, we denote
\begin{align*}
    \ep \coloneqq \ep_F \quad \text{and}\quad
    R \coloneqq R_{F,p}.
\end{align*}}\black{We begin by giving an explicit construction of the Shintani set: 
\black{
\begin{lemma} \label{shin_set_construction}
    The Shintani set $R$ can be written as:
        \begin{align*}
           R &= 
           %R(D_{K/F}) =
                %\Bigg\{ \frac{A}{tp} + \frac{B}{tp}\ep \;\ : \;A + s B \equiv 0 \pmod{t}, \; A,B\in \Z,\; A\in (0,tp], B\in [0, tp ) \Bigg\}.\\
                \Bigg\{ \frac{A}{tp} + \frac{B}{tp}\ep \;\ : \;A + s B \equiv 0 \pmod{t}, \;  A\in (0,tp] \cap \Z, B\in [0, tp ) \cap \Z \Bigg\}.
        \end{align*}
    \end{lemma}
    \begin{proof}
        By Lemma \ref{pof}, we have $D_{K/F} = p\cao_F.$ As such, for any element $r_1+r_2\ep \in \frac{1}{p}\cao_F,$ we have 
            \begin{align*}
                r_1 + r_2 \ep = (r_1 + s r_2) + t r_2\theta_F \in \frac{1}{p}\cao_F. % \implies  r_1+s r_2 = \frac{m}{p} \text{ and } t r_2 = \frac{B}{p}\quad\textrm{ for some }m,B\in \Z.
            \end{align*}
   \black{The set $\{1, \theta_F\}$ constitutes an integral basis of $\cao_F$, meaning we can write any element of the Shintani set as $r_1+r_2\ep = \frac{A'}{p} +  \frac{B}{p}\theta_F \in \frac{1}{p}\cao_F$, for some $A', B \in \Z$. Note that
   \begin{align*}
       \frac{A'}{p} = {r_1+sr_2}\quad \text{and}\quad  \frac{B}{p} = tr_2.
   \end{align*}
   In particular, we have
   \begin{align*}
       r_2 = \frac{B}{tp},
   \end{align*}
   and since $r_2 \in [0,1)$, we see that $B \in [0, tp)$. Additionally, we see that 
   \begin{align*}
       r_1 = \frac{A'}{p}-sr_2 = \frac{tA'-sB}{tp} = \frac{A}{tp}
   \end{align*}
    where $A\coloneqq tA'-sB.$ We know that $r_1 \in (0, 1]$, so $A \in (0, tp]$. Moreover, since $$\frac{A}{tp}+\frac{B}{tp}\ep=\frac{A+sB}{tp}+\frac{B}{p}\theta_F\in\frac{1}{p}\cao_F,$$ we must also have $A+sB\equiv0\pmod{t}.$ From the expression above, we can see that every element of the form $\frac{A}{tp}+\frac{B}{tp}\ep$ with $A,B\in\Z$, $A\in(0,tp]$, $B\in[0,tp)$, $A+sB\equiv0\pmod{t}$ is in the Shintani set. This finishes the proof.}
    \end{proof}}}
\black{Next, we want to identify $R$ with the finite field $\F_{p^2}$. We begin with the work of Barquero-Sanchez, Masri, and Tsai, who proved that $R$ is a finite abelian group with respect to the following operation:
\begin{align*}
    r \oplus r' \coloneqq r + r' + \Z[\ep]
\end{align*}
(see Proposition 4.3 in \cite{wei_lun_stark_units}). This allows us to prove the following proposition relating $R$ to $\F_{p^2}$, a property that is central to our proof of Theorem \ref{thm:1.1}.
\begin{prop}\label{pi_prop}
    The Shintani set $R$ has a structure as a $\Z[\ep]$-module. This structure admits a surjective $\Z[\ep]$-module homomorphism $\pi: R\to \F_{p^2}.$
\end{prop}}
\black{\begin{proof}
We begin with the $\Z[\ep]$-module structure on $R$. By definition, the fractional ideal $\frac{1}{p}\cao_F$ is an $\cao_F$-module, and since $\Z[\ep]$ is a subring of $\cao_F$, we observe that $\frac{1}{p}\cao_F$ is a $\Z[\ep]$-module by restriction of scalars. Furthermore, since $\Z[\ep]$ is a $\Z[\ep]$-submodule of $\frac{1}{p}\cao_F$, we have that $\tfrac{1}{p}\cao_F/\Z[\ep]$ is a $\Z[\ep]$-module. Moreover, $R\subset\tfrac{1}{p}\cao_F$ is a complete reduced set of coset representatices for $\tfrac{1}{p}\cao_F/\Z[\ep]$ (see Proposition 4.1 in \cite{wei_lun_stark_units}). Thus $R$ has the structure of a $\Z[\ep]$-module and can be identified with $\tfrac{1}{p}\cao_F/\Z[\ep].$

Since $R \subseteq \frac{1}{p}\cao_F$, multiplication by $p$ defines an injective $\Z[\ep]$-module homomorphism $$R \longrightarrow \cao_F, \hspace{0.5cm} r \longmapsto pr+p\cao_F.$$ If we compose this map with the projection $\cao_F \longrightarrow \cao_F/p\cao_F,$ we obtain the map $$\pi: R \longrightarrow \cao_F/p\cao_F, \hspace{0.5cm} r \longmapsto pr,$$ which is surjective, as shown in Proposition 4.4 of \cite{wei_lun_stark_units}. If $p$ remains inert in $F$, then $p\cao_F$ is prime and thus maximal, so $\cao_F/p\cao_F$ is a finite field. Then, since $F$ is quadratic and $\cao_F=\Z[\theta_F]$, we know that
\begin{align*}
    \cao_F/p\cao_F \cong \Z[\theta_F]/p\Z[\theta_F] \cong \F_{p}[\theta]\cong\F_{p^2}.
\end{align*}
\end{proof}}
\black{\begin{remark}
By the First Isomorphism Theorem, $$R/\ker(\pi) \cong \cao_F/p\cao_F.$$ 
Note that this is an isomorphism of groups, and therefore pertains only to the structure of $R$ as an additive abelian group. We do not require a multiplicative structure within $R$ here; rather we point out that the map $$\pi: R/\ker(\pi) \to \F_{p^2}$$ is bijective. We will make use of this bijective correspondence in the proof of Theorem ~\ref{thm:1.1}.
\end{remark}}
\black{\begin{lemma}\label{kernel_of_phi_lemma} 
The elements of $\ker(\pi)$ are exactly those elements of $R$ which are in $\cao_F.$
\end{lemma}
\begin{proof}
Assume $r \in \ker(\pi)$. Then, since $\Im(\pi) = \cao_F/p\cao_F$, we have 
\begin{align*}
    \pi(r) =0 \iff pr \in p\cao_{F} \iff r\in \cao_{F}. 
\end{align*}\end{proof}}

\black{Now we are in a position to explicitly describe the elements in $\ker(\pi)$.}
\black{\begin{lemma}\label{explicit_kernel}
    The kernel of the map $\pi$ is given by
    \begin{align*}
        \ker(\pi) = \left\{1  - \bigg\{\frac{si}{t}\bigg\}_{[0,1)}+ \frac{i}{t}\ep\;|\; 0\leq i \leq t-1\right\}
    \end{align*}
    %where $spi \pmod{tp}$ denotes the least non-negative residue of $spi$ modulo $tp.$ 
    In particular, we have that $|\ker(\pi)| = t$.
\end{lemma}
\begin{proof}
    Consider $r\in\ker(\pi).$ Using Lemma \ref{shin_set_construction} and noting that $\ep=s+t\theta_F,$ we see that $r$ has the form
    \begin{align*}
        r=\frac{A}{tp}+\frac{B}{tp}\ep=\frac{A+sB}{tp}+\frac{B}{p}\theta_F.
    \end{align*}
    By Lemma \ref{kernel_of_phi_lemma}, $r\in\ker(\pi)\iff r\in R\cap\cao_F,$ so we have that $\tfrac{A+sB}{tp}\in\Z$ and $\tfrac{B}{p}\in\Z.$ The second condition implies $p|B,$ and since $B\in[0,tp)\cap\Z$ by Lemma \ref{shin_set_construction}, we see that $B=pi$ for $i\in[0,t)\cap\Z.$ The condition that $\tfrac{A+sB}{tp}\in\Z$ implies that $A\equiv -sB=spi\pmod{tp}.$ Since $A\in(0,tp]\cap\Z$ by Lemma \ref{shin_set_construction}, $A$ is uniquely determined by $B.$ More precisely,
    \begin{align*}
        A=tp-(spi\pmod{tp})
    \end{align*} where $spi\pmod{tp}$ is the least positive residue of $spi\in\Z$ modulo $tp.$ We can further simplify this expression; since $$spi \pmod{tp} = spi - tp\bigg\lfloor \frac{spi}{tp}\bigg\rfloor,$$ we have that 
    \begin{align*}
        \ker(\pi)\subseteq\left\{1  - \bigg\{\frac{si}{t}\bigg\}_{[0,1)}+ \frac{i}{t}\ep \;:\; i\in[0,t)\cap\Z\right\},
    \end{align*}
     where $\{\cdot\}$ denotes the fractional part function $\{x\}_I$, defined as the unique element of $I$ satisfying $x-\{ x\}_I \in \Z$.  The converse containment is seen immediately from the fact that that $\ep=s+t\theta_F$ and the definition of $spi\pmod{tp}.$ Thus, $\ker(\pi)$ has size exactly $t.$
\end{proof}}\black{\noindent We will denote the elements of $\ker(\pi)$ as $$\kappa_i \coloneqq 1  - \bigg\{\frac{si}{t}\bigg\}_{[0,1)}+ \frac{i}{t}\ep \quad\text{for}\quad i\in\{0,1,\ldots,t-1\}.$$}
\black{
\section{Proof of Theorem 1.1}\label{section_1.1_proof}
Equipped with these facts about the Shintani set described in the previous section, we now prove Theorem ~\ref{thm:1.1}. Our proof relies on features of the structure of the Shintani set which come from from the bijection between $R_{F,p}/\ker(\pi)$ and $\F_{p^2}$, as well as some properties we derive of the Hecke character across the Shintani set. Again, to \black{simplify} notation, we let $\ep = \ep_F$, $R =R_{F,p}$, and $\rho = \rho_{F,p}.$}

\black{
We are now able to describe the Shintani set using the multiplicative structure of $\F_{p^2}^\times = \langle \rho +p\cao_F\rangle$. Using the bijection from $R/\ker(\pi)$ to $\F_{p^2}$, we have
\begin{align*}
    R = \ker(\pi) \sqcup \left(\bigsqcup_{m=1}^{p^2-1} \pi^{-1}(\rho^m+p\cao_F)\right).
\end{align*}
For each $m$ between $1$ and $p^2-1$, choose one element in the coset $\pi^{-1}(\rho^m+p\cao_F)$, which we denote $\tilde{x}(m) + \tilde{y}(m)\ep \in R$. 

Next, we explicitly calculate each $\tilde{x}(m)$ and $\tilde{y}(m)$ in terms of $\rho^m$. Note that $\{1, \theta_F\}$ is a $\F_{p}$-basis of $\F_{p^2}$, so we can write $\rho^m+p\cao_F \coloneqq x(m) + y(m)\theta_F$ for some integers $x(m), y(m)$. Observe that, since $\ep = s + t\theta_F$, we have
\begin{align*}
    \rho^m+p\cao_F = x(m) - \frac{s\cdot y(m)}{t} + \frac{y(m)}{t}\ep.
\end{align*}
Under multiplication by $p$ and reduction modulo $p$, the point $$\frac{x(m)}{p} - \frac{s\cdot y(m)}{tp} + \frac{y(m)}{tp}\ep \in \frac{1}{p}\cao_F$$ maps to $\rho^m+p\cao_F$. Thus, if we subtract a suitable element of $\Z[\ep]$ from this point, we obtain a point $\tilde{x}(m) + \tilde{y}(m)\ep \in R$ that is a preimage of $\pi^{-1}(\rho^m+p\cao_F)$. In particular, we see that
\begin{align*}
    \tilde{x}(m) = \left\{\frac{{x(m)}}{p} - \frac{s\cdot y(m)}{tp}\right\}_{(0,1]}, \hspace{0.5cm} \tilde{y}(m) = \left\{\frac{y(m)}{tp}\right\}_{[0,1)}.
\end{align*}
 
Since $\pi(\tilde{x}(m) + \tilde{y}(m)\ep) = \rho^m+p\cao_F$, we can construct the entire coset from this element:
\begin{align*}
    \pi^{-1}(\rho^m+p\cao_F) = \bigg\{(\tilde{x}(m) + \tilde{y}(m)\ep) \oplus \kappa_i : 1 \le i \le t\bigg\}.
\end{align*}
For simplicity, we write 
\begin{align}
    \tilde{x}_i(m) + \tilde{y}_i(m)\ep \coloneqq (\tilde{x}(m) + \tilde{y}(m)\ep) \oplus \kappa_i.
\end{align}
Using our explicit construction of $\ker(\pi)$ given in Lemma ~\ref{explicit_kernel}, we can similarly explicitly construct each $\tilde{x}_i(m)$, $\tilde{y}_i(m)$. We see that \black{
\begin{align*}
    \tilde{x}_i(m) &= \left\{\frac{x(m)}{p} - \frac{s\cdot y(m)}{tp} + 1  - \bigg\{\frac{si}{t}\bigg\}_{[0,1)} \right\}_{(0,1]}\\
    \tilde{y}_i(m) &= \left\{\frac{y(m)}{tp}+ \frac{i}{t}\right\}_{[0,1)}.
\end{align*}}
Thus, we can write $R$ as the following disjoint union:
\begin{align*}
     R = \ker(\pi) \sqcup \left(\bigsqcup_{m=1}^{p^2-1}\bigg\{\tilde{x}_i(m) + \tilde{y}_i(m)\ep : 1 \le i \le t\bigg\}\right).
\end{align*}}\black{By Proposition~\ref{shintani_simplified}, we simplify Shintani's class number formula to obtain
\begin{align}\label{thm1_unsimplified}
    h_{K} = \frac{1}{2}\sum_{\substack{1\leq m\leq p^2-1\\ 1\leq i \leq t}} \chi_{K/F}\bigg((\tilde{x}_i(m) + \tilde{y}_i(m)\ep) \cdot p\cao_F\bigg)
    \cdot\sum_{\substack{0\leq l_1,l_2\leq2\\ l_1+l_2=2}}{\frac{B_{l_1}(\tilde{x}_i(m))}{l_1!}\frac{B_{l_2}(\tilde{y}_{i}(m))}{l_2!}\tre^{l_2-1}}.
\end{align}}

\black{We can also simplify the Hecke character term. Consider any element $r_1+r_2\ep\in R$ and any element $k_1+k_2\ep \in \ker(\pi)$. By Lemma ~\ref{kernel_of_phi_lemma}, we have $k_1+k_2\ep\in\cao_F.$ Since the Hecke character has conductor $p\cao_F,$ we have
\begin{align*}      \chi_{K/F}\bigg((r_1+r_2\ep+k_1+k_2\ep)p\cao_F\bigg)=\chi_{K/F}\bigg(\left((r_1+r_2\ep)p\cao_F\right)+p\cao_F\bigg) =\chi_{K/F}\bigg((r_1+r_2\ep)p\cao_F\bigg).
\end{align*}
Thus, the value of $\chi_{K/F}(r_1 + r_2\ep)$ depends only on the coset of $r_1 + r_2\ep$ in $R/\ker(\pi)$. \black{
Therefore, we have that 
\begin{align*}    \chi_{K/F}\bigg((\tilde{x}_i(m)+\tilde{y}_i(m)\ep)p\cao_F\bigg) &= \chi_{K/F}\bigg((\tilde{x}(m)+\tilde{y}(m)\ep)p\cao_F\bigg).
\end{align*}
By definition, $$p(\tilde{x}(m)+\tilde{y}(m)\ep)-\rho^m \in p\cao_F.$$
% we have that 
% \begin{align*}
%     (\tilde{x}(m)+\tilde{y}(m)\ep)p = \rho^m+p\cao_F.
% \end{align*}
Using this and the multiplicativity of the Hecke character, we get
\begin{align*}   \chi_{K/F}((\tilde{x}_i(m)+\tilde{y}_i(m)\ep)p\cao_F) = \chi_{K/F}((\rho^m+p\cao_F)\cao_F) = \chi_{K/F}((\rho+p\cao_F)\cao_F)^m.
\end{align*}

If $\chi_{K/F}((\rho+p\cao_F)\cao_F) = 0$, then since $(\cao_F/p\cao_F)^{\times} = \langle \rho+p\cao_F \rangle$, we would have that $\chi_{K/F}(rp\cao_F)=0$ for all $r\in R.$ However, this contradicts the definition of $\chi_{K/F}$. Moreover, $\chi_{K/F}((\rho+p\cao_F)\cao_F) \neq 1,$ since we would then similarly have that  $\chi_{K/F}(rp\cao_F)=1$ for all $r\in R\setminus\cao_F,$ but $\chi_{K/F}$ is a non-trivial character by construction. Thus, we see that $\chi_{K/F}((\rho+p\cao_F)\cao_F) = -1,$ which implies
\begin{align*}
\chi_{K/F}\bigg((\tilde{x}_i(m)+\tilde{y}_i(m)\ep)p\cao_F\bigg) = \chi_{K/F}((\rho+p\cao_F)\cao_F)^m=(-1)^m.
\end{align*}}Thus, Equation \ref{thm1_unsimplified} simplifies further:
\begin{align*}
    h_{K} = \frac{1}{2}\sum_{\substack{1\leq m\leq p^2-1\\1\leq i \leq t}}(-1)^{m}\sum_{\substack{0\leq l_1,l_2\leq2\\ l_1+l_2=2}}{\frac{B_{l_1}(\tilde{x}_i(m))}{l_1!}\frac{B_{l_2}(\tilde{y}_i(m))}{l_2!}\tre^{l_2-1}}.
\end{align*}
\black{Next, we simplify the Bernoulli polynomial part of the class number formula. We consider 
\begin{align*}
    &\sum_{\substack{0\leq l_1,l_2\leq2\\ l_1+l_2=2}}{\frac{B_{l_1}(\tilde{x}_i(m))}{l_1!}\frac{B_{l_2}(\tilde{y}_i(m))}{l_2!}\trep^{l_2-1}} \\& \hspace{1 cm}= \trep\frac{\tilde{x}_i(m)^2-\tilde{x}_i(m)+1/6}{2} + 2\left(\tilde{x}_i(m)-\frac{1}{2}\right)\left(\tilde{y}_i(m)-\frac{1}{2}\right) +\trep\frac{\tilde{y}_i(m)^2-\tilde{y}_i(m)+1/6}{2}\\
    &\hspace{1 cm}= \frac{\trep}{2} \left(\tilde{x}_i(m) - \frac{1}{2}\right)^2 + 4\left(\tilde{x}_i(m) - \frac{1}{2}\right)\left(\tilde{y}_i(m) - \frac{1}{2}\right)  + \frac{\trep}{2} \left(\tilde{y}_i(m) - \frac{1}{2}\right)^2 + c_0,
\end{align*}
for some constant $c_0$. Since for any constant $c$, $$\sum_{\substack{1\leq m \leq p^2-1\\ 1\leq i \leq t}}(-1)^m \cdot c =0,$$ 
we can ignore the constant term $c_0$ that arises in the inner sum of Bernoulli polynomials.} We can write the class number $h_K$ as 
\begin{equation}
\label{longscaryeqn}
    h_{K} = \frac{1}{4}\sum_{\substack{1\leq m\leq p^2-1\\1\leq i \leq t}}(-1)^{m} \Bigg[\trep \left(\tilde{x}_i(m) - \frac{1}{2}\right)^2 + 4\left(\tilde{x}_i(m) - \frac{1}{2}\right)\left(\tilde{y}_i(m) - \frac{1}{2}\right)
    + \trep \left(\tilde{y}_i(m) - \frac{1}{2}\right)^2\Bigg].
\end{equation}\black{If we define
\begin{equation}
\label{tilde_defns}
    x_i(m) \coloneqq tp(2\tilde{x}_i(m)-1) \quad\text{and}\quad y_i(m) \coloneqq tp(2\tilde{y}_i(m)-1),
\end{equation}
we can then rewrite the above equation as
\begin{align}
\label{actualscaryeqn}
    h_{K} &= \frac{1}{16t^2p^2}\sum_{\substack{1\leq m \leq p^2-1\\ 1\leq i\leq t}}(-1)^m \bigg[\tr_{F/\Q}(\ep)\Big({x}_i(m)\Big)^2 + 4\Big({x}_i(m)\Big)\Big({y}_i(m)\Big) + \tr_{F/\Q}(\ep)\Big({y}_i(m)\Big)^2\bigg].
\end{align}
Finally, by defining the quadratic form $$Q_{F}(Y_1, Y_2)\coloneqq \tr_{F/\Q}(\ep)Y_1^2+4Y_1Y_2+\tr_{F/\Q}(\ep)Y_2^2,$$ we rewrite \eqref{actualscaryeqn} as 
\begin{equation}
 \begin{aligned}
     h_K=\frac{1}{16t^2 p^2}\sum_{\substack{1\leq m\leq p^2-1\\ 1\leq i\leq t}}(-1)^mQ_{F}\left({x}_i(m),{y}_i(m)\right).
 \end{aligned}
 \end{equation}}}\black{The last step is to derive recurrence relations for $x(m), y(m)$, \black{the coefficients of $\rho^m +p\cao_F = x(m) + y(m)\theta_F$}. The minimal polynomial of $\theta_F$ is $x^2 - \tr(\theta_F)x + \norm_{F/\Q}(\theta_F)$, which implies $$\theta_F^2 = \tr_{F/\Q}(\theta_F)\theta_F - \norm_{F/\Q}(\theta_F).$$ To simplify notation, let $T=\tr_{F/\Q}(\theta_F)$ and $N=\norm_{F/\Q}(\theta_F)$. Since $\rho\coloneqq a+b\ep,$ we have the initial conditions $x(1) = a$ and $y(1) = b$. Then, we get
\begin{align*}
    \rho^{m+1} = x(m+1) + y(m+1)\theta_F &= \big(x(m) + y(m)\tf\big)\cdot \big(a+b\theta_F\big)\\
    &= a\cdot x(m)-Nb\cdot y(m)+\bigg(b\cdot x(m)+\bigg(a+Tb\bigg)\cdot y(m)\bigg)\theta_F.
\end{align*}}\black{This implies the following recurrence relations:
    \begin{align*}
        x(m+1) &= a\cdot x(m)-Nb\cdot y(m)\\
        y(m+1) &= b\cdot x(m)+(a+Tb)\cdot y(m).
    \end{align*}
Then, consider functions $X(z)$, $Y(z)$ given by
\begin{align}
        X(z) = \sum_{m=1}^\infty x(m)\cdot z^m,\quad Y(z) = \sum_{m=1}^\infty y(m)\cdot z^m.
    \end{align}
Using our recurrence relations, we can set up a system of equations to find explicit expressions for $X(z), Y(z)$ as rational functions determined by $x(1)$ and $y(1)$. We see that
\begin{align}\notag
        X(z) &= z\cdot \left[a\cdot X(z) - N b\cdot Y(z)\right] + az\\\notag
        Y(z) &= z\cdot \left[ b\cdot X(z) + \bigg(a + Tb\bigg)\cdot Y(z)\right] + bz,
\end{align}
which gives
\begin{align}\notag
    X(z) &= \dfrac{az - (a^2 + abT + Nb^2)z^2}{(a^2 + abT + Nb^2)z^2 - (2a + bT)z + 1} \\\notag
    Y(z) &= \dfrac{bz}{(a^{2} + abT + b^{2}N)z^{2} - (2a + bT)z + 1} .
\end{align}
We simplify these by letting $C_{F,p} \coloneqq a^{2} + abT + Nb^{2}$ and $D_{F,p} \coloneqq 2a + bT$ to get
\begin{align}\notag
    X(z) = \dfrac{az - C_{F,p}z^2}{C_{F,p}z^2 - D_{F,p}z + 1},\quad Y(z) =  \dfrac{bz}{C_{F,p}z^{2} - D_{F,p}z + 1}.
\end{align}}\black{Note that the coefficients $x(m), y(m)$ of the power series of these rational functions correspond to those $x(m), y(m)$ which we use to generate each $x_i(m), y_i(m)$ using the formulas \begin{equation*}
\begin{aligned}
    x_i(m) \coloneqq tp(2\tilde{x}_i(m)-1) \quad\text{and}\quad y_i(m) \coloneqq tp(2\tilde{y}_i(m)-1).
\end{aligned}
\end{equation*}This concludes the proof of Theorem ~\ref{thm:1.1}.
}
\black{
\section{Proof of Theorem 1.2}\label{section4}
Here we prove Theorem $~\ref{thm 1.2},$ which relies heavily on the structure of the Shintani set as a $\Z[\ep_F]$-module and the related attributes of the base-$\ep_F$ expansions of its elements. Through a series of preliminary lemmas, we set up the proof of Theorem ~\ref{thm 1.2} by relating the base-$\ep_F$ expansion of $1/p$ to the orbit of elements in $R_{F,p}\setminus \ker(\pi)$ under the action of $\ep_F$. This allows us to derive a finite sum analogous to Girstmair's \eqref{GirstmairFormula}, in which the number of summands is equal to the period length of the base $\ep_F$ expansion of $1/p$.

\black{
Thoughout this section, we fix a totally real quadratic field $F$ and an imaginary quadratic extension $K \coloneqq F(\sqrt{-p})$, where $p \equiv 3\pmod{4}$ and $p$ remains inert in $\cao_F$. To \black{simplify} notation, we also let
\begin{align*}
    \ep \coloneqq \ep_F \quad\text{and}\quad
    R \coloneqq R_{F,p}.
\end{align*} Additionally, we denote $r \in R$ as $r \coloneqq r_1 + r_2\ep$.
}\black{
\subsection{Shintani Cycles}
Recall from Section ~\ref{sec2} that we can identify $R$ with $\frac{1}{p}\cao_F/\Z[\ep]$ to make it into a $\Z[\ep]$ module. In particular, the multiplicative group $\langle \ep \rangle$ acts on $\frac{1}{p}\cao_F$ via scalar multiplication. If we denote the map for this group action by $$\mu: \langle \ep \rangle \times \frac{1}{p}\cao_F \longrightarrow \frac{1}{p}\cao_F,$$ we can compose $\mu$ with the projection map $$\nu: \frac{1}{p}\cao_F \longrightarrow \frac{1}{p}\cao_F/\Z[\ep]$$ to yield $$\mu' \coloneqq \nu\circ\mu: \langle \ep \rangle \times \frac{1}{p}\cao_F \longrightarrow \frac{1}{p}\cao_F/\Z[\ep].$$ Note that, since $\nu$ is a $\Z[\ep]$-module homomorphism, $\mu'$ constitutes a group action of $\langle\ep\rangle$ on $\frac{1}{p}\cao_F/\Z[\ep].$
% Recall from Section ~\ref{sec2} that $R \cong \frac{1}{p}\cao_F/\Z[\ep]$ is a $\Z[\ep]$-module homomorphism. We define the projection map $$\nu: \frac{1}{p}\cao_F \longrightarrow \frac{1}{p}\cao_F/\Z[\ep],$$ which is a surjective $\Z[\ep]$-module homomorphism. We can see that $\langle \ep \rangle$ acts on $\frac{1}{p}\cao_F$ by $\ep \cdot x = \ep x$. This is because $\langle \ep \rangle \subset \Z[\ep]$, and $\frac{1}{p}\cao_F$ is a $\Z[\ep]$-module, so the group action axioms are satisfied by the module axioms. If we define the map for this group action as $$\mu: \langle \ep \rangle \times \frac{1}{p}\cao_F \longrightarrow \frac{1}{p}\cao_F,$$ we can compose $\mu$ with $\nu$ to yield $$\mu' \coloneqq \nu\circ\mu: \langle \ep \rangle \times \frac{1}{p}\cao_F \longrightarrow \frac{1}{p}\cao_F/\Z[\ep].$$ Note that, since $\nu$ is a surjective $\Z[\ep]$-module homomorphism, $\nu\circ\mu$ constitutes a group action of $\langle\ep\rangle$ on $\frac{1}{p}\cao_F/\Z[\ep].$
\begin{lemma} The map
    \begin{align*}
    \overline{\mu}: \langle \ep \rangle \times \frac{1}{p}\cao_F/\Z[\ep] \longrightarrow \frac{1}{p}\cao_F/\Z[\ep], \hspace{0.3 cm}
    (\ep, \alpha + \Z[\ep]) \longmapsto \mu'(\ep, \alpha)
    \end{align*}is a well-defined group action.
\end{lemma}
\begin{proof}
    Since $\mu'$ is a group action, it is sufficient to show that $\overline{\mu}$ is well-defined. 
     Take $\alpha, \alpha' \in \frac{1}{p}\cao_F$ such that $\alpha + \Z[\ep] = \alpha' + \Z[\ep]$. Thus, for any $n \in \Z$, we have
     \begin{align*}
         \overline{\mu}(\ep^n,\alpha + \Z[\ep]) &= \mu'(\ep^n,\alpha + \Z[\ep])\\
         &= \ep^n\alpha + \Z[\ep] = \ep^n\alpha' + \Z[\ep]\\
         &= \overline{\mu}(\ep^n, \alpha' + \Z[\ep]).
     \end{align*}
     Thus, $\overline{\mu}$ is well-defined, so $\overline{\mu}$ constitutes a group action of $\langle \ep \rangle$ on $\frac{1}{p}\cao_F/\Z[\ep]$.
\end{proof}

We know from Proposition 4.1 in \cite{wei_lun_stark_units} that $R$ is a complete and reduced set of representatives of $\frac{1}{p}\cao_F/\Z[\ep]$. Thus, we obtain a group action $\langle\ep\rangle \curvearrowright R$ given by 
$(\ep,r) \mapsto \ep\ast r,$ where $\ep\ast r = \ep r + z,$ and $z$ is the unique element of $\Z[\ep]$ such that $\ep r + z \in R$. \black{We define the \emph{Shintani cycle} of any element $r\in R$ to be the orbit of $r$ under this action, and we denote this set as $C_r \coloneqq \langle \ep \rangle \ast r.$ }}
\black{\begin{remark}
Note that $\ep \ast r \in \cao_F \iff r \in \cao_F$. Thus, for any $r \in R \cap \cao_F$, every element in the Shintani cycle of $r$ is an element of $\cao_F$. We will call Shintani cycles containing elements in $R \setminus \cao_F$ the \emph{nontrivial Shintani cycles} of $R$. \black{We refer to Shintani cycles of elements in $R\cap\cao_F$ as \emph{trivial Shintani cycles} because the elements in these cycles are weighted by a factor of $0$ in Shintani's class number formula (see the remark in Section \ref{thm1.2proof_subsection}), and hence for our purposes are ``trivial."}
\end{remark}}

\subsection{Epsilon Expansions}\label{ep_exp_section}
A base-$\ep$ expansion is an analogue to the usual decimal expansion. The base-$\ep$ expansion of any element $\alpha\in F$ is computed in the following way. Let $n\coloneqq\lfloor\log_{\ep}(\alpha)\rfloor.$ Then we have $$\alpha=a_n\ep^n+a_{n-1}\ep^{n-1}+\ldots+a_0+a_{-1}\ep^{-1}+\ldots,$$where 
\begin{align*}
    a_n\coloneqq\lfloor\alpha/\ep^n\rfloor,
    a_{n-1}\coloneqq\lfloor(\alpha-a_n\ep^n)/\ep^{n-1}\rfloor,\ldots, a_i\coloneqq\lfloor(\alpha-a_n\ep^n-\ldots-a_{i+1}\ep^{i+1})/\ep^{i}\rfloor, \ldots
\end{align*}
We observe that $\ep$ is an algebraic integer which is real since $F$ is real quadratic, and that $\ep$ must be $>1$ since it is a totally positive fundamental unit. Moreover, $\ep$ must have Galois conjugate with absolute value $<1$ since $F$ is a real quadratic field and $\ep$ has norm $1$. Thus, $\ep$ is a Pisot number by definition, and by consequence, Theorem $3.1$ in \cite{pisot} shows that any element of $R$ has an eventually periodic base-$\ep$ expansion. 

\black{
For some $\alpha \in F$ whose base $\ep$ expansion can be written as $$\alpha = a_n\ep^n+\ldots+a_0+a_{-1}\ep^{-1}+\ldots+a_{-k}\ep^{-k}+\overline{a_{-k-1}\ep^{-k-1} + \ldots + a_{-k-P_\alpha}\ep^{-k-P_\alpha}},$$
we call $P_\alpha$ the \emph{period length} of the base-$\ep$ expansion of $\alpha$. Additionally, we will call the ordered set $$\{a_{-k-1}, \ldots, a_{-k-P_\alpha}\}$$ the \emph{period set} of the base-$\ep$ expansion of $\alpha$.} We can further observe that any element of $F$ whose base-$\ep$ expansion is finite is an element of $\Z[\ep]$, by the following argument.
\begin{lemma}\label{finite_ep_exp}
    If $\alpha \in F$ has a finite base-$\ep$ expansion, then $\alpha \in \Z[\ep]$.
\end{lemma}
\begin{proof}
If $\gamma$ has a finite base-$\ep$ expansion, we can express it as
    \begin{align*}
        \gamma=\sum_{i=K_1}^{K_2}m_i\ep^i
    \end{align*}
    where $K_1, K_2$ are integers. Using that $\ep^2=\trep\ep-1$ and that $\ep^{-1}=\trep-\ep$, we can perform the following replacement on any term $m_k\ep^k$ where $k\neq 0$ or $1$:
    \begin{align*}
        m_k\ep^k = \begin{cases}
            m_k(\trep\ep-1)^{k/2} & \text{if }k \text{ is an even positive integer}\\
            m_k(\trep\ep-1)^{(k-1)/2}\ep & \text{if }k \text{ is an odd positive integer}\\
            m_k(\trep-\ep)^{-k} & \text{if }k \text{ is a negative integer}.
        \end{cases}
    \end{align*}
    The third equality implies that any negative integer power of $\ep$ can be converted to a linear combination of positive integer powers of $\ep$. Therefore, it suffices to show that a linear combination of positive integer powers of $\ep$ can be expressed as an element of $\Z[\ep]$. The first two equalities guarantee that any positive power of $\ep$ can be expressed a strictly lower positive power of $\ep$. Thus by induction, any finite linear combination of (possibly negative) powers of $\ep$ can be expressed as an element of $\Z[\ep]$.
\end{proof}
\begin{prop}\label{same_repeating_part}
The repeating part in the base-$\ep$ expansion of any two elements in the same Shintani cycle is the same. 
\end{prop}
To prove this lemma, we require some preliminaries. Consider some $r \in R$, where $r = r_1 + r_2\ep$, and recall that $0 < r_1 \leq 1$ and $0 \leq r_2 < 1$, $r_1, r_2 \in \Q$. Hence the action of $\ep$ on $R$ amounts to:
\newline
\begin{align*}
    \ep\ast r &= \ep\cdot(r_1+r_2\ep)+z_1+z_2\ep
\end{align*}
where $z_1+z_2\ep$ is the unique element in $\Z[\ep]$ such that $\ep\cdot(r_1+r_2\ep)+z_1+z_2\ep\in R$. We can explicitly compute bounds for $z_1$ and $z_2$:
\black{\begin{lemma}
\label{prop:abbounds}
    If $\ep\ast(r_1+r_2\ep)=\ep\cdot(r_1+r_2\ep)+z_1+z_2\ep$, then 
    \begin{align*}
    z_1 = 1,\quad\text{and}\quad z_2 = -\lfloor r_1+r_2\emph{Tr}_{F/\Q}(\ep)\rfloor.
    \end{align*}
\end{lemma}
\begin{proof}
    The minimal polynomial of $\ep$ is $$x^2 - \tr_{F/\Q}(\ep)x + \norm_{F/\Q}(\ep) = x^2 - \trep x + 1,$$ and thus $$\ep^2 = \trep\ep - 1.$$
    Consider 
    \begin{align*}
        \ep(r_1 + r_2\ep) = r_1\ep + r_2\ep^2 = -r_2 + (r_1+r_2\trep)\ep.
    \end{align*} 
    To find $\ep\ast(r_1+r_2\ep)$, we must shift $\ep(r_1+r_2\ep)$ by some $z_1+z_2\ep \in \Z[\ep]$ such that 
    \begin{align*}
        -r_2 + z_1 \in (0,1],\quad\text{and}\quad r_1+r_2\trep + z_2 \in [0,1).
    \end{align*}
    \black{It is immediately apparent that $z_2 = -\lfloor r_1+r_2\trep \rfloor.$ Further, since $r_1+r_2\ep \in R$, we have that $r_2 \in [0,1)$, so we see that $z_1 = 1$.}
\end{proof}
Note that the above proposition and our bounds on $r_1$ and $r_2$ imply that $-\trep \leq z_2 \leq 0$.
}Before we show that the repeating part of the base $\ep$ of elements in the same Shintani cycle is the same, we require one more fact about $\ep$, which we now prove.
\begin{lemma}
\label{prop:eps=trace}
    We have that $\left\lceil\ep\right \rceil=\emph{Tr}_{F/\Q}(\ep)$.
\end{lemma}
\begin{proof} 
Let $\ep=s+t\sqrt{d}$. We start by showing $\trep\geq\left\lceil \ep\right\rceil$. Observe that 
\begin{align*}
    \trep = \ep+\frac{1}{\ep}
    \implies \trep>\ep
    \implies \trep\geq\left\lceil \ep\right\rceil,
\end{align*}
because the trace of an algebraic integer is always an element of $\Z$.

\black{Now we will show that $\trep \leq \lceil \ep \rceil$. Assume for the sake of contradiction that $\trep \ge \lceil \ep \rceil + 1$. Then we have that
\begin{align*}
    \trep \ge \lceil \ep \rceil + 1 \implies \frac{\ep^2+1}{\ep} \ge \lceil\ep\rceil + 1  \implies 1 - \ep \ge \ep\lceil\ep\rceil - \ep^2 \ge 0 
    \implies 1 \ge \ep.
\end{align*} 
However, by definition, $\ep > 1$, so we see that $$\trep < \lceil\ep\rceil + 1 \implies \trep \le \lceil\ep\rceil.$$}
\end{proof}
\black{
Now we proceed to prove that the repeating part in the base-$\ep$ expansion of any two elements in the same Shintani cycle is the same. 
\begin{proof}[Proof of Proposition \ref{same_repeating_part}]
     Consider some element $r\in R,$ with base-$\ep$ expansion
     \begin{align*}
         r &= a_1\ep + a_0+ a_{-1} \ep^{-1} + a_{-2} \ep^{-2} + a_{-3}\ep^{-3} + \ldots
     \end{align*}
     Note that since $r\in R,$ $\lfloor\log_{\ep}(r)\rfloor = 0\text{ or } 1,$ so the highest power of $\ep$ appearing in the base-$\ep$ expansion of $r$ is  at most $1.$ Given this base-$\ep$ expansion of $r,$ we have that
     \begin{equation}\label{ep_star_r}
         \ep\ast r = \ep\cdot r + z_2\ep +1 = a_{1}\ep^2 + (a_0 + z_2)\ep^1 +( a_{-1}+1) \ep^0 + a_{-2}\ep^{-1} + a_{-3}\ep^{-2} + \ldots
     \end{equation} Recall that in a base-$\ep$ expansion, each digit (in this case $a_i$ for $i\in\Z$) must be an element of the set $A\coloneqq\{0,1,\ldots, \lfloor\ep\rfloor\}.$ We now consider the following two cases: in Case $1,$ both $a_0 + z_2$ and $a_{-1}+1$ are in $A$; in Case $2,$ one or both of $a_0+z_2$ and $a_{-1}+1$ is not in $A.$ 
     
     \textbf{Case 1.}
     In Case $1,$ the expression in \eqref{ep_star_r} is already a valid base-$\ep$ expansion of $\ep\ast r.$ We can see that only a finite number of digits differ between the base-$\ep$ expansion of $\ep\ast r$ and the base-$\ep$ expansion of $r,$ so in this case the repeating part of $\ep\ast r$ must be the same as $r.$
     
    \textbf{Case 2.}
     Now we address Case $2,$ which we can split into Case $2.1$ and Case $2.2.$ In Case $2.1,$ $a_0+z_2\not\in  A;$ in Case $2.2,$ $a_{-1}+1\not\in A.$
     
    \textbf{Case 2.1.}    
     Assume that $a_0+z_2\not\in  A.$ Since $a_0$ is a digit in the base-$\ep$ expansion of $r$, $0\leq a_0\leq\lfloor\ep\rfloor$ by definition. Additionally, $-\trep\leq z_2\leq0$ by Lemma ~\ref{prop:abbounds}. Thus it always true that $-\trep\leq a_0+z_2\leq \trep.$ Therefore if $a_0+z_2\not\in  A,$ it must be that $-\trep\leq a_0+z_2\leq -1.$ Then we have that $0\leq a_0+z_2+\trep\leq\trep-1,$ so $a_0+z_2+\trep\in A$ and is hence an acceptable digit. Since $\trep=\ep+\ep^{-1},$ we can rewrite \eqref{ep_star_r} as 
     \begin{align}\notag
         \ep\ast r &= (a_{1}-1)\ep^2 + (a_0 + z_2+\trep)\ep + (a_{-1}+1-1) + a_{-2}\ep^{-1} + a_{-3}\ep^{-2} +\ldots\\ \label{2.1_shift}
         &= (a_{1}-1)\ep^2 + (a_0 + z_2+\trep)\ep + a_{-1} + a_{-2}\ep^{-1} + a_{-3}\ep^{-2} +\ldots
     \end{align}
    Since $0\leq a_0+z_2+\trep<\trep,$ the above base-$\ep$ expansion is valid as long as $0\leq a_{1}-1\leq\lfloor\ep\rfloor.$
    Since $0\leq a_1\leq\lfloor\ep\rfloor,$ we know $-1\leq a_1-1\leq\lfloor\ep\rfloor-1.$ Thus unless $a_1-1=-1,$ it must be true that $0\leq a_{1}-1\leq\lfloor\ep\rfloor.$ Let us assume for the sake of contradiction that $a_1-1=-1.$ If we let
    \begin{align*}
        \alpha&=\ep^2\text{ and }\\
        \beta&=(a_0 + z_2+\trep)\ep + a_{-1} + a_{-2}\ep^{-1} + a_{-3}\ep^{-2} +\ldots,
    \end{align*}
    then \eqref{2.1_shift} implies that $\ep\ast r=-\alpha+\beta.$ However, it follows directly from the definition of a base-$\ep$ expansion \black{and the fact that $\ep$ is a Pisot number} that $\alpha>\beta.$ Thus $\ep\ast r =-\alpha+\beta<0,$ which contradicts the fact that $\ep\ast r\in R.$ Thus we have that $0\leq a_{1}-1\leq\lfloor\ep\rfloor,$ so \eqref{2.1_shift} is a valid base-$\ep$ expansion of $\ep\ast r.$ We can see that only a finite number of digits differ between the base-$\ep$ expansion of $\ep\ast r$ and the base-$\ep$ expansion of $r.$ Therefore the repeating part of $\ep\ast r$ must be the same as $r.$ 
    
    \textbf{Case 2.2}
    Assume $a_{-1}+1\not\in A.$ By Case $2.1,$ we may assume without loss of generality that $a_0+z_2\in A.$ Since $0\leq a_{-1}\leq\lfloor\ep\rfloor,$ we know that $1\leq a_{-1}+1\leq\lfloor\ep\rfloor+1.$ Thus if $a_{-1}+1\not\in A,$ it must be that $a_{-1}+1=\lfloor\ep\rfloor+1,$ so $a_{-1}=\lfloor\ep\rfloor.$ Since $\lceil\ep\rceil=\trep$ by Lemma ~\ref{prop:eps=trace}, we have $a_{-1}+1=\trep.$ Again using that $\trep=\ep+\ep^{-1},$ we can rewrite \eqref{ep_star_r} as
    \begin{equation}\label{2.2_shift}
        \ep\ast r = a_{1}\ep^2 + (a_0 + z_2+1)\ep + (a_{-2}+1)\ep^{-1} + a_{-3}\ep^{-2} + \ldots
    \end{equation}
    Since $0\leq a_0 + z_2\leq\lfloor\ep\rfloor$ by assumption, if $a_0 + z_2+1\not\in A,$ then $a_0+z_2+1=\lfloor\ep\rfloor+1.$ This would imply that $\ep\ast r>\ep+1,$ which contradicts the fact that $\ep\ast r \in R.$ Thus it must be that $0\leq a_0 + z_2+1 \leq\lfloor\ep\rfloor.$ With this, we see that if $a_{-2}+1\in A,$ then \ref{2.2_shift} is a valid base-$\ep$ expansion of $\ep\ast r.$ Otherwise, if $a_{-2}+1\not\in A,$ then since $a_{-2}\in A,$ it must be that $a_{-2}=\lfloor\ep\rfloor,$ so $a_{-2}+1=\lceil\ep\rceil=\trep.$ Using that $\trep=\ep+\ep^{-1},$ we can rewrite \ref{2.2_shift} as 
    \begin{equation}
        \ep\ast r = a_{1}\ep^2 + (a_0 + z_2+1)\ep + 1 + (a_{-3}+1)\ep^{-2} + \ldots
    \end{equation}
    We note that by the same argument used before, if $a_{i}+1\not\in A$ for any $i\in\Z,$ then $a_i=\lfloor\ep\rfloor=\trep - 1.$ Thus if we let $j$ be the smallest positive integer such that $a_{-j}\neq\lfloor\ep\rfloor,$ then continuing in the same manner, we see that 
    \begin{equation}
        \ep\ast r = a_{1}\ep^2 + (a_0 + z_2+1)\ep + 1 + \ep^{-1} + \ep^{-2} + \ldots + \ep^{-j+3}+(a_j+1)\ep^{-j+1}+a_{j+1}\ep^{-j}\ldots
    \end{equation}
    Since the base-$\ep$ expansion of $r$ must be finite or periodic, it is certainly possible to choose such an index $j.$ We can assume that $r$ does not have repeating part $\overline{\lfloor\ep\rfloor},$ since $\lfloor\ep\rfloor\ep^{i}+\lfloor\ep\rfloor\ep^{i-1}+\lfloor\ep\rfloor\ep^{i-2}+\ldots=\ep^{i+1}$ for any $i\in\Z.$ Thus in Case $2.2,$ we see that only a finite number of digits differ between the base-$\ep$ expansion of $\ep\ast r$ and the base-$\ep$ expansion of $r.$ Therefore the repeating part of $\ep\ast r$ must be the same as $r$ in this case.
    
    Now we have seen that in all cases, the repeating part of the base-$\ep$ expansion of $\ep\ast r$ is the same as that of $r,$ which finishes the proof.
\end{proof}}}

\black{In many of the results which follow, it will prove useful for us to note the following fact about the map $\pi$ as defined in Proposition \ref{pi_prop}.
\begin{lemma}\label{equivariance}
    The map $\pi$ is equivariant under the action of $\langle\ep\rangle.$
\end{lemma}
\begin{proof}
    Since $\frac{1}{p}\cao_F$ is an ideal of $\cao_F,$ it is also a $\cao_F$-module. Moreover, $\cao_F$ is trivially an $\cao_F$-module, and $p\cao_F$ is a submodule of $\cao_F$ since $p\cao_F$ is an ideal of $\cao_F$. Thus the map $$\phi:\frac{1}{p}\cao_F\to\cao_F/p\cao_F$$ defined by multiplication by $p$ is an $\cao_F$-module homomorphism. The kernel of this map is $\cao_F,$ so by the first isomorphism theorem, $\frac{1}{p}\cao_F/\cao_F\xrightarrow{\sim}\cao_F/p\cao_F$ is an isomorphism of $\cao_F$-modules. Since $\Z[\ep]$ is a subring of $\cao_F,$ by restriction of scalars, $\frac{1}{p}\cao_F/\cao_F\cong\cao_F/p\cao_F$ is also an isomorphism of $\Z[\ep]$-modules. 

    We can also observe that since $\cao_F$ is a submodule of $\frac{1}{p}\cao_F,$ the projection map $$\frac{1}{p}\cao_F/\Z[\ep]\to\frac{1}{p}\cao_F/\Z[\ep]\bigg/\cao_F/\Z[\ep]$$ is a surjective $\Z[\ep]$-module homomorphism. By the third isomorphism theorem, we further have that $$\frac{1}{p}\cao_F/\Z[\ep]\bigg/\cao_F/\Z[\ep]\cong\frac{1}{p}\cao_F/\cao_F,$$ which implies $$\psi: \frac{1}{p}\cao_F/\Z[\ep]\to\frac{1}{p}\cao_F/\cao_F$$ is a surjective $\Z[\ep]$-module homomorphism. Since $R\subset \frac{1}{p}\cao_F$ and $R$ constitutes a complete set of coset representatives for $\frac{1}{p}\cao_F/\Z[\ep]$ (see \cite{wei_lun_stark_units}, Proposition $4.1$), the identity map $$\iota: R\to \frac{1}{p}\cao_F/\Z[\ep]$$ is a $\Z[\ep]$-module isomorphism. Now we can see that since $\pi=\phi\circ\psi\circ\iota,$ $\pi$ is a surjective $\Z[\ep]$-module homomorphism. Thus $\pi$ is equivariant under the action of $\langle\ep\rangle.$
\end{proof}}
\black{With this result, we may now examine more closely the action of $\ep$ on $R.$ Namely, we can deduce the following fact about nontrivial Shintani cycles. 
\begin{lemma}\label{same_cycle_length}
All nontrivial Shintani cycles have length equal to the multiplicative order of $\ep+p\cao_F$ in $(\cao_F/p\cao_F)$. 
\end{lemma}
\begin{proof}
Let $M$ denote the multiplicative order of $\ep+p\cao_F$ in $\cao_F/p\cao_F$, and consider $r\in R\setminus\cao_F.$ We will show that $|C_r| = M.$ Suppose that $\ep^m\ast r =r$ for some $m\in\Z^+.$ Then by Lemma ~\ref{equivariance}, 
\begin{align*}
    \pi(r)=\pi(\ep^m\ast r)=(\ep^m+p\cao_F)\pi(r)
    \implies (\ep^m-1+p\cao_F)\pi(r)=0.
\end{align*}
By assumption $r\not\in\cao_F,$ so $\pi(r)\neq 0$ by Lemma \ref{kernel_of_phi_lemma}. Since $\cao_F/p\cao_F$ is a field, it must be that $(\ep^m-1+p\cao_F)=0.$ Therefore $\ep^m\equiv 1 \mod{p\cao_F},$ so $M | m.$ As a consequence, if we denote the stabilizer subgroup associated to $r$ under the action of $\ep$ as $\langle\ep\rangle_r,$ then $\langle\ep\rangle_r\subset\langle\ep^M\rangle.$ Moreover, we note that since $\langle\ep\rangle_r$ is a subgroup of $\langle\ep^M\rangle$, it must be that $\langle\ep\rangle_r=\langle\ep^{M'}\rangle$ for some $M'$ such that $M|M'.$ Moreover, we also have that
\begin{align*}
    \ep^{M'}\ast r=r\implies (\ep^{M'}-1+p\cao_F)\pi(r)=0
\end{align*}
where $\pi(r)\neq0,$ so that $\ep^{M'}\equiv 1 \mod{p\cao_F}.$ Thus it must also be that $M'|M,$ so $M'=M,$ and hence $\langle\ep\rangle_r=\langle\ep^M\rangle.$ In other words, $\langle\ep^M\rangle$ is the stabilizer subgroup of $r$ for all $r\in R\setminus\cao_F$. By the orbit-stabilizer theorem, we then have that $|C_r|=[\langle\ep\rangle:\langle\ep^M\rangle]=M.$
\end{proof}}
In the lemma which follows, we equate $M$ with the minimal period length of the base-$\ep$ expansion of $r$ for all $r\in R\setminus\cao_F.$ \black{This fact, combined with Lemma \ref{same_cycle_length}, will then imply that for all $r\in R\setminus\cao_F$, $|C_r|=P_r,$ where $P_r$ denotes the minimal period length of the base-$\ep$ expansion of $r$.}
\begin{lemma}\label{orbit_length_is_cycle_length}
     \black{For any $r\in R\setminus \cao_F$, the minimal period length of the base-$\ep$ expansion of $r$ is equal to the multiplicative order of $\ep+p\cao_F$ in $(\cao_F/p\cao_F)$.}
\end{lemma}
\begin{proof}
Consider an element $r = r_1+r_2\ep \in R\setminus\cao_F.$ We start by showing that $P_r | M$. As mentioned at the beginning of Section \ref{ep_exp_section}, the base-$\ep$ expansion of $r$ is always eventually periodic, say
\begin{align}\label{ep_exp_sum}
    r=\sum_{i=-1}^{N-1}{a'_i\ep^{-i}}+\ep^{-N}\sum_{j=0}^\infty {\left(a_1\ep^{-jP_r}+a_2\ep^{-jP_r-1}+\ldots+a_{P_r}\ep^{-jP_r-P_r+1}\right)}.
\end{align}
We remind the reader that since $r\in R, \lfloor\log_{\ep}(r)\rfloor = 0\text{ or } 1,$ so the highest power of $\ep$ in \eqref{ep_exp_sum} is $1.$ Multiplying \eqref{ep_exp_sum} by $\ep^{P_r}$, we obtain 
\begin{align*}
\ep^{P_r}r&=\sum_{i=-1}^{N-1}{a'_i\ep^{-i+P_r}}+\ep^{-N}\sum_{j=0}^\infty {\left(a_1\ep^{-(j-1)P_r}+a_2\ep^{-(j-1)P_r-1}+\ldots+a_{P_r}\ep^{-(j-1)P_r-P_r+1}\right)}.
\end{align*}
Reindexing \eqref{ep_exp_sum}, we get
\begin{align*}
    r=\sum_{i=-1}^{N-1}{a'_i\ep^{-i}}+\ep^{-N}\sum_{j=1}^\infty {\left(a_1\ep^{-(j-1)P_r}+a_2\ep^{-(j-1)P_r-1}+\ldots+a_{P_r}\ep^{-(j-1)P_r-P_r+1}\right)}.
\end{align*} And thus 
\begin{align*}
    \ep^{P_r}r-r=\left(\sum_{i=-1}^{N-1}{a'_i\ep^{-i+P_r}-a'_i\ep^{-i}}\right)+\ep^{-N}\left(a_1\ep^{P_r}+a_2\ep^{P_r-1}+\ldots+a_{P_r}\ep\right).
\end{align*}
Let
\begin{align*}
    \alpha = \sum_{i=-1}^{N-1}{a'_i\ep^{-i+P_r}}, \quad\text{and}\quad\beta = -\sum_{i=-1}^{N-1}{a'_i\ep^{-i}}, \quad\text{and}\quad\gamma = \ep^{-N}\left(a_1\ep^{P_r}+a_2\ep^{P_r-1}+\ldots+a_{P_r}\ep\right).
\end{align*}
Note that, because $\alpha$, $\beta$, and $\gamma$ have finite $\ep$ expansions, we have that $\alpha, \beta, \gamma \in \Z[\ep]$, and thus $$ \alpha + \beta + \gamma = \ep^{P_r}r-r \in \Z[\ep].$$ By definition, we know that $$\ep^{P_r}\ast r = \ep^{P_r}r + z$$ for some $z \in \Z[\ep]$. Thus, using Lemmas ~\ref{kernel_of_phi_lemma} and ~\ref{equivariance}, we have
\begin{align*}
    \ep^{P_r}\ast r-r-z= \ep^{P_r}r - r \implies \pi(\ep^{P_r}\ast r-r-z) = \pi(\ep^{P_r}r - r) 
    \implies (\ep^{P_r} + p\cao_F - 1)\pi(r) = 0.
\end{align*}
Since $F$ is a field in which $\pi(r)\neq 0$ since $r\not\in\cao_F,$ we have that 
\begin{align*}
    \ep^{P_r} - 1 + p\cao_F = 0 \implies \ep^{P_r} \equiv 1 \pmod{p\cao_F}.
\end{align*}
for any $r \in R\setminus\cao_F$. Recall that $M$ is the multiplicative order of $\ep$ in $\cao_F/p\cao_F$, so we see that $M|P_r$.

Next, we show that $P_{r} | M$. Since both $r$ and $\ep \ast r$ have periodic base-$\ep$ expansions, we let $N_1$ represent the smallest integer such that the repeating part of the base-$\ep$ expansion of $r$ begins in the $\ep^{-N_{1}}$ place. Similarly, let $N_2$ represent the smallest integer such that the repeating part of the base-$\ep$ expansion of $\ep\ast r$ begins in the $\ep^{-N_{2}}$ place. 

Let $S = \max(N_1, N_2)$ be the smallest integer such that the base-$\ep$ expansion of both $r$ and $\ep\ast r$ is periodic for all indices greater than $S$. Thus, the digits in the $\ep^{-S}, \ep^{-S-1}, \ldots, \ep^{-S-P_r}$ place of the base-$\ep$ expansion of $r$ constitute a full period, and we let the ordered set 
\begin{align*}
    \{x_1, x_2, \ldots, x_{P_{r}}\}
\end{align*}
represent the {period set} of $r$. As shown in Proposition ~\ref{same_repeating_part}, the operation $\ep\ast r$ shifts the digits within the repeating part of the base-$\ep$ of $r$ to the left by one index. In other words, the period set of $\ep\ast r$ is the ordered set
\begin{align*}
    \{x_2, \ldots, x_{P_{r}}, x_1\}.
\end{align*}
Note that moving between the period set of $r$ and the period set of $\ep\ast r$ can be represented by applying the permutation
\begin{align*}
    \tau = (1 \hspace{0.2 cm} 2 \hspace{0.2 cm} \cdots \hspace{0.2 cm} P_r) \in S_{P_r}
\end{align*}
to the period set of $r.$ Additionally, we have that $r = \ep^{M}\ast r$, so the period sets of $r$ and $\ep^{M}\ast r$ must be equal. Thus, 
\begin{align*}
    \tau^{M} \{x_1, x_2, \ldots, x_{P_{r}}\} = \{x_1, x_2, \ldots, x_{P_{r}}\},  
\end{align*}
which implies that $\tau^{M}$ is the identity permutation. Since the order of $\tau \in S_{P_r}$ is $P_r$, we have that $P_r | M$. So, we see that $P_r = M$.
\end{proof}

\black{In our final lemma before we prove Theorem \ref{thm 1.2}, we show that for any nontrivial Shintani cycle, the sum of the coefficients $r_1$ and $r_2$ where $r=r_1+r_2\ep$ of all the elements $r$ in the Shintani cycle is a constant. In fact, these coefficients sum to $M.$}
\black{
\begin{lemma}\label{sum_of_a_cycle_is_l}
    For any $r\in R\setminus\cao_F,$ let $r'\coloneqq r_1'+r_2'\ep$. Then,
    \begin{align*}
        \sum_{r'\in C_r}{(r_1'+r_2')}=M.
    \end{align*}  
\end{lemma}
\begin{proof}
    Let $\ep^i\ast r \coloneqq r_1(i)+r_2(i)\ep,$ so $r_1(i+1)+r_2(i+1)\ep=\ep\ast (r_1(i)+r_2(i)\ep).$ Recall that \begin{align*}
        \ep\ast (r_1(i)+r_2(i)\ep)=(1-r_2(i))+\{r_1(i)+\trep r_2(i)\}\ep
    \end{align*}
    by Lemma \ref{prop:abbounds}. Comparing coefficients, we see that $r_1(i+1)+r_2(i)=1$ for all $i\in\Z$. Moreover, since $M=|C_r|$ for all $r\in R\setminus \cao_F$ by Lemma ~\ref{same_cycle_length}, we have that $r_1(m)=r_1(m+M)$ for any integer $m.$ Using these facts, we see that
    \begin{align*}
        \sum_{r\in C_r}{(r_1(i)+r_2(i))}&=\sum_{i=1}^M {(r_1(i)+r_2(i))}=r_1(1)+r_2(M)+\sum_{i=1}^{M-1} {r_1(i+1)}+\sum_{j=1}^{M-1}{r_2(j)} \\  &=r_1(M+1)+r_2(M)+\sum_{i=1}^{M-1} {\bigg(r_1(i+1)}+r_2(i)\bigg)=1+(M-1)=M.
    \end{align*}  
\end{proof}}
\black{
\subsection{Proof of Theorem 1.2}\label{thm1.2proof_subsection}
Since $\langle \ep \rangle$ acts on $R$, $R$ decomposes into a disjoint union of Shintani cycles, under this action. Letting $\mathcal{L}$ denote a complete reduced set of Shintani cycle representatives for $R,$ and recalling that $C_r$ denotes the Shintani cycle of $r,$ we can rewrite Shintani's formula as follows:
\begin{align}\notag
h_K &= \frac{1}{2}\sum_{r\in R}{\chi_{K/F}(rp\cao_F)\sum_{\substack{0\leq l_1,l_2\leq2\\ l_1+l_2=2}}{\frac{B_{l_1}(r_1)}{l_1!}\frac{B_{l_2}(r_2)}{l_2!}\trep^{l_2-1}}}\\ \label{eqn:begin_thm_1.2}
&= \frac{1}{2}\sum_{i=1}^{|C_r|}\sum_{r\in\mathcal{L}}{\chi_{K/F}(\ep_{F}^i\ast r \cdot p\cao_F)\sum_{\substack{0\leq l_1,l_2\leq2\\ l_1+l_2=2}}{\frac{B_{l_1}(r_1)}{l_1!}\frac{B_{l_2}(r_2)}{l_2!}\trep^{l_2-1}}}
\end{align}

First, we show that $\chi_{K/F}(r' \cdot p\cao_F)$ is constant for all $r'\in C_r$. By definition of $\ep\ast r$, we see that $\ep\ast r=\ep r +z$ for some $z \in \cao_F.$ Thus
\begin{align*}
    \chi_{K/F}(\ep^i\ast r \cdot p\cao_F) &= \chi_{K/F}( (\ep^ir + z) \cdot p\cao_F)= \chi_{K/F}(\ep^irp\cao_F + zp\cao_F ).
\end{align*}
Since $zp\cao_F\subset p\cao_F,$ $p\cao_F | zp\cao_F,$ and since $p\cao_F$ is the conductor of this Hecke character, we see that 
\begin{align*}
    \chi_{K/F}( \ep^i\ast r \cdot p\cao_F) 
    &= \chi_{K/F}(\ep^irp\cao_F ).
\end{align*}
Additionally, $\ep$ is a unit, so we know $$r\cao_F = \ep^ir \cao_F $$ for any integer $i.$ Therefore
$$\chi_{K/F}(\ep^i\ast r \cdot p\cao_F)= \chi_{K/F}(\ep^ir p\cao_F)= \chi_{K/F}( rp \cao_F),$$ and thus the Hecke character value in \eqref{eqn:begin_thm_1.2} is constant throughout each Shintani cycle.  

By Lemma ~\ref{same_cycle_length}, all nontrivial Shintani cycles in $R$ contain the same number of elements. Since $1/p$ is an element of $R\setminus\cao_F$ by Lemma ~\ref{orbit_length_is_cycle_length}, the period length $\ell_{F,p}$ of the base-$\ep$ expansion of $1/p$ is equal to the length of each nontrivial cycle. \black{
\begin{remark} \label{trivial_cycles}
    Note that, for all $r\in R\cap\cao_F,$ the Hecke character $\chi_{K/F}(rp\cao_F)$ evaluates to 0, so elements $r\in R\cap\cao_F$ are all weighted by a factor of $0$ in \eqref{eqn:begin_thm_1.2}. Hence, we can ignore them in our calculations.
\end{remark} }
Using these facts, we obtain
\begin{equation}\label{thm1.2_simplification_intermediate_step}
    h_K=\frac{1}{2}\sum_{i=1}^{\ell_{F,p}}\hspace{+.1 cm}\sum_{r\in\mathcal{L}}{\chi_{K/F}(rp\cao_F)\sum_{\substack{0\leq l_1,l_2\leq2\\ l_1+l_2=2}}{\frac{B_{l_1}(r_1)}{l_1!}\frac{B_{l_2}(r_2)}{l_2!}\trep^{l_2-1}}}.
\end{equation}
As shown in \cite{wei_lun_stark_units}, we have that
\begin{align*}
    \sum_{r \in R} \chi_{K/F}( r p\cao_F) = \sum_{r \in R\cap\cao_F} \chi_{K/F}( r p\cao_F) + \sum_{r \in R \setminus \cao_F} \chi_{K/F}( r p\cao_F) = 0.
\end{align*}
\black{For all $r \in \cao_F,$ we have already seen that $\chi_{K/F}( rp\cao_F) = 0.$ } Thus,
\begin{align*}
    0 &= \sum_{r \in R\setminus\cao_F} \chi_{K/F}( r p\cao_F) = \sum_{\black{r \in \mathcal{L}\setminus\cao_F}} \chi_{K/F}( r p\cao_F)\cdot\ell_{F,p} = \ell_{F,p}\sum_{\black{r \in \mathcal{L}\setminus\cao_F}} \chi_{K/F}(r p\cao_F),
    \end{align*}
which yields $$\sum_{r \in \mathcal{L}\setminus\cao_F} \chi_{K/F}( r p\cao_F)=0.$$ In other words, we have character orthogonality across the elements $r \in \mathcal{L}\setminus\cao_F$. With this, we consider the sum over Bernoulli polynomials within this formula. \black{Letting \begin{align*}
    \mathcal{B}(r_1+r_2\ep) \coloneqq \sum_{\substack{0\leq l_1,l_2\leq2\\ l_1+l_2=2}} \frac{B_{l_1}(r_1)B_{l_2}(r_2)}{l_1!l_2!} \tr_{F/\Q}\left(\ep^{l_2-1} \right),
\end{align*}we see that
\begin{align}\notag
    \mathcal{B}(r_1+r_2\ep) &= \dfrac{r_1^2-r_1+\frac{1}{6}}{2}\trep + 2\bigg(r_1 - \frac{1}{2}\bigg)\bigg(r_2 - \frac{1}{2}\bigg) + \dfrac{r_2^2-r_2+\frac{1}{6}}{2}\trep\\\label{eqn:thm_1.2_middle}~~
    &= \frac{\trep}{2}\bigg(r_1^2 + r_2^2 - (r_1+r_2) + \frac{1}{3}\bigg) + 2r_1r_2 - (r_1+r_2) + \frac{1}{2}.
\end{align}}% \begin{align}\notag
%     \sum_{\substack{0\leq l_1,l_2\leq2\\ l_1+l_2=2}}{\frac{B_{l_1}(r_1)}{l_1!}\frac{B_2(r_2)}{l_2!}\trep^{l_2-1}} = \dfrac{r_1^2-r_1+\frac{1}{6}}{2}\trep + 2\bigg(r_1 - \frac{1}{2}\bigg)\bigg(r_2 - \frac{1}{2}\bigg) + \dfrac{r_2^2-r_2+\frac{1}{6}}{2}\trep\\\label{eqn:thm_1.2_middle}~~
%     = \frac{\trep}{2}\bigg(r_1^2 + r_2^2 - (r_1+r_2) + \frac{1}{3}\bigg) + 2r_1r_2 - (r_1+r_2) + \frac{1}{2}.
% \end{align}\christy{need to fix the eq tag above, it looks weird but idk how to fix}
Recall that by Lemmas ~\ref{sum_of_a_cycle_is_l}, for all $r\in\mathcal{L}\setminus\cao_F,$ $$\sum_{r'\in C_r}r_1'+r_2'=M.$$ Thus, we can further simplify \eqref{eqn:thm_1.2_middle} to
\begin{equation}\label{thm1.2_simplification_second_intermediate_step}
    \mathcal{B}(r_1+r_2\ep)=\frac{\trep}{2}\bigg(r_1^2 + r_2^2 - M + \frac{1}{3}\bigg) + 2r_1r_2 - M - \frac{1}{2}.
\end{equation}
Because we have character orthogonality over $\mathcal{L}\setminus\cao_F,$ 
we can add a constant to the inner Bernoulli sum of \eqref{thm1.2_simplification_intermediate_step} without changing the value of the whole expression. In particular, if we let 
\begin{align*}
    c\coloneqq\frac{\trep}{2}\bigg(-M+\frac{1}{3}\bigg) - M-\frac{1}{2},
\end{align*}
we see that \eqref{thm1.2_simplification_second_intermediate_step} can be rewritten as 
\begin{align*}
    \mathcal{B}(r_1+r_2\ep)&= \frac{\trep}{2}\bigg(r_1^2 + r_2^2\bigg) + 2r_1r_2 + c\\
    &=\frac{1}{2}\bigg(\trep r_1^2 + 4r_1r_2 +\trep r_2^2\bigg) + c.
\end{align*} Using these results and letting $\ep\ast r\coloneqq r_1(i)+r_2(i)\ep,$ we obtain
\begin{align*}\label{eqn:1.2stinky}
    h_K = \frac{1}{4}\sum_{i=1}^{\ell_{F,p}}\sum_{r\in \mathcal{L}\setminus\cao_F}{\chi_{K/F}}\bigg(r p\cao_F\bigg)\bigg(\trep r_1(i)^2 + 4r_1(i)r_2(i) + \trep r_2(i)^2 \bigg).
\end{align*}}\black{Recall that $Q_F(Y_1,Y_2)\coloneqq\trep Y_1^2+4Y_1Y_2+\trep Y_2^2.$ Thus if we make a slight abuse of notation by letting $Q_F(\ep^i\ast r)=\trep (r_1(i)^2+4r_1(i)r_2(i)+\trep r_2(i)^2,$ then we can express $h_K$ as 
\begin{align*}
    h_K = \frac{1}{4}\sum_{i=1}^{\ell_{F,p}}\sum_{r\in \mathcal{L}\setminus\cao_F}{\chi_{K/F}}(r p\cao_F)Q(\ep^i\ast r).
\end{align*}}
\black{\section{Examples}\label{example_section}
Here we illustrate Theorems \ref{thm:1.1} and \ref{thm 1.2} for $\Q(\sqrt{3},\sqrt{-p})$, where $p$ is prime. Note that the ring of integers of $F=\Q(\sqrt{3})$ is given by $\Z[\sqrt{3}]$, and its totally positive unit group $\cao_{F}^{\times,+}$ is generated by $\ep_{F} = 2 + \sqrt{3}$. We require that $p\equiv 3\pmod{4}, \legendre{3}{p} = -1,$ and $7\leq p.$ The first two conditions imply that the relative discriminant ideal is the prime ideal $p\Z[\sqrt{3}]$. Consequently, $\Z[\sqrt{3}]/p\Z[\sqrt{3}] ~\cong ~\F_p[\sqrt{3}]$.  

\subsection{Theorem 1.1 with $F=\Q(\sqrt{3})$}
Let $\rho_{F,p} = a + b\sqrt{3}$ be a generator of $\F_p[\sqrt{3}]$. Table \ref{table1} lists values of $\rho_{F,p}$ as computed with {\tt{SageMath}}. Using these values, we use \eqref{def:c_defns} and \eqref{def:d_defns} to calculate $C_{F,p}$ and $D_{F,p}$, then use \eqref{eqn:x1defns} and \eqref{eqn:x2defns} to find the corresponding rational functions $X_{F,p}(z)$ and $Y_{F,p}(z)$, which are also displayed in Table \ref{table1}.

We extract the first $p^2-1$ coefficients from our rational functions by taking the $k^{th}$ derivative of $X(z)$ and $Y(z)$, evaluating each function at $z=0$, and dividing by $k!$. Note that in this case, since $t=1$, we obtain only $1$ sequence $x_1(m),$ and $y_1(m),$ from each of $X_{F,p}(z)$ and $Y_{F,p}(z)$ respectively. Since $\tr_{F/Q}(\ep_F)=4$, we have
\begin{align*}
    Q_{F}(Y_1, Y_2)=4Y_1^2+4Y_1Y_2+4Y_2^2.
\end{align*}
Now we may apply Theorem ~\ref{thm:1.1} to obtain
\begin{align*}
    h_{F(\sqrt{-p})}=\frac{1}{16p^2}\sum_{1\leq m\leq p^2-1}(-1)^mQ_{F}\bigg(x_1(m),y_1(m)\bigg).
\end{align*}
\black{The smallest suitable prime for which we can apply Theorem ~\ref{thm:1.1} here is $p=7,$ for which we calculate
\begin{align*}
    h_{F(\sqrt{-7})}=\frac{1}{784}(&-84+76-300+52-28+436-100+148-196+52-108+124-84+148\\
    &-36+172-28+124-12+76-196+172-4+156-84+76-300+52-28\\
    &+156-100+316-196+52-108+124-84+316-36+228-28+124-12\\
    &+76-196+228-4+436)=2.
\end{align*}}In Table \ref{table1}, we list some terms of our alternating sum for the class numbers of all such primes less than $100$, along with the corresponding class numbers calculated using Theorem~\ref{thm:1.1} and verified using {\tt{SageMath}}. 
\bgroup
\def\arraystretch{3.0}
\black{
\begin{table}[ht]
    \centering
\begin{tabular}{|c | c | c | c | c |}
 \hline
 $\hspace{0.25 cm} p \hspace{0.25 cm} $ & $\hspace{0.25 cm}\rho_{F,p} \hspace{0.25 cm}$ & $\hspace{0.4 cm}X_{F,p}(z) \hspace{0.4 cm}$ & $\hspace{0.4 cm}Y_{F,p}(z) \hspace{0.4 cm}$ & $\hspace{1.0 cm}$$h_{F_(\sqrt{-p})}$ Calculation $\hspace{1.0 cm}$ \\  
 \hline\hline
 7 &$ 6+\sqrt{3}$ & $\dfrac{6z - 33z^2}{33z^2 - 12z + 1}$ & $\dfrac{z}{33z^{2} - 12z + 1}$ & $\frac{1}{784}\bigg( -84 + 76 - \ldots + 436\bigg) = 2$\\ 
 \hline
 19 & $1+4\sqrt{3}$ & $\dfrac{6z +47z^2}{-47z^2 - 12z + 1}$ & $\dfrac{4z}{-47z^{2} - 12z + 1}$ & $\frac{1}{5776} \bigg(-364 + 252 - \ldots + 3892 \bigg) = 2$\\
 \hline
 31 & $1+6\sqrt{3}$ & $\dfrac{z +107 z^2}{-107z^2 - 2z + 1}$ & $\dfrac{6z}{-107z^{2} - 2z + 1}$ & $\frac{1}{15376} \bigg(-1084 + 676 - \ldots + 10804 \bigg) = 6$ \\
 \hline
 43 & $1+5\sqrt{3}$ & $\dfrac{z +74z^2}{-74z^2 - 2z + 1}$ & $\dfrac{5z}{-74z^{2} - 2z + 1}$ & $\frac{1}{29584} \bigg(-3556 + 4836 - \ldots + 21172 \bigg) = 6$ \\
 \hline
 67 & $2+5\sqrt{3}$ & $\dfrac{2z + 71z^2}{-71z^2 - 4z + 1}$ & $\dfrac{5z}{-71z^{2} - 4z + 1}$ & $\frac{1}{71824}\bigg(-11772 + 2212 - \ldots + 52276\bigg) = 6$\\  
 \hline
 79 & $2+6\sqrt{3}$ & $\dfrac{2z + 104z^2}{-104z^2 - 4z + 1}$ & $\dfrac{6z}{-104z^{2} - 4z + 1}$ & $\frac{1}{99856} \bigg(-16068 + 7372 - \ldots + 73012\bigg) = 30$\\ [1ex] 
 \hline
\end{tabular}
\caption{Theorem \ref{thm:1.1} for primes $p<100.$}
\label{table1}
\end{table}}
\subsection{Theorem 1.2 with $F = \Q(\sqrt{3})$}
We illustrate Theorem ~\ref{thm 1.2} in the same setting. Letting $F=\Q(\sqrt{3})$, we calculate $h_{K}$ for $p\equiv 3 \pmod{4}$ where $7\leq p$ and $\legendre{3}{p}=-1$. We remind the reader that $\ep_F=2+\sqrt{3}$, so $t=1$. Thus by Lemma ~\ref{explicit_kernel}, $\ker(\pi)=R_{F,p}\cap\cao_F=\{1\}$.

In the case that $p=7,$ we first calculate the base-$\ep_F$ expansion of $1/7,$
\begin{align*}
    \frac{1}{7}&=\ep_F^{-2}+3\ep_F^{-3}+2\ep_F^{-4}+2\ep_F^{-6}+2\ep_F^{-7}+3\ep_F^{-8}+3\ep_F^{-11}+2\ep_F^{-12}+2\ep_F^{-14}+2\ep_F^{-15}+3\ep_F^{-16}+\ldots\\
    &=0.01\overline{32202230}.
\end{align*}
Noticing that $1/7$ has period length $\ell_{F,7}=8$, by Lemma $~\ref{orbit_length_is_cycle_length}$, we can then deduce that there are \begin{align*}
    \frac{|R_{F,7}\setminus \cao_F|}{\ell_{F,7}}=\frac{1\cdot 7^2-1}{8}=6
\end{align*}
disjoint Shintani cycles which comprise $R_{F,7}\setminus \cao_F$. We can generate these Shintani cycles explicitly, by calculating $\ep_F^i \ast r$ for $0\leq i < 8$ for $r\in R_{F,7}\setminus\cao_F$. One can verify that 
\begin{align*}
    \mathcal{L}=\left\{\frac{1}{7}+\frac{1}{7}\ep_F,\hspace{+0.1 cm}\frac{1}{7},\hspace{+0.1 cm}\frac{1}{7}+\frac{4}{7}\ep_F,\hspace{+0.1 cm}\frac{1}{7}+\frac{5}{7}\ep_F,\hspace{+0.1 cm}\frac{2}{7}+\frac{2}{7}\ep_F,\hspace{+0.1 cm}\frac{3}{7}\right\}
\end{align*}
is a complete reduced set of representatives for all $6$ distinct nontrivial cycles in $R_{F,7}.$} With these values, we now calculate $h_{F(\sqrt{-7})}$ using Theorem $~\ref{thm 1.2}$. Noting that $\tre=4$ so $Q_F(Y_1,Y_2)=4Y_1^2+4Y_1Y_2+4Y_2^2,$ we compute 
\begin{align*}
    h_{F(\sqrt{-7})}&=\frac{1}{4}\sum_{i=1}^{8}\hspace{+0.23cm}\sum_{r\in \mathcal{L}} \chi_{F(\sqrt{-7})/F} \left(r {p}\cao_{F}\right) \bigg(4 r_1(i)^2 + 4r_1(i)r_2(i) + 4r_2(i)^2 \bigg)\\
    &=\frac{1}{4}\bigg(-\frac{220}{7} + \frac{228}{7} - \frac{188}{7} + \frac{212}{7} - \frac{180}{7} + \frac{204}{7}\bigg)=2.
\end{align*}
In Table \ref{table_thm1.2}, we carry out the same procedure for all suitable primes less than $100.$
\bgroup
\def\arraystretch{3.0}
\begin{table}[ht]
    \centering
\begin{tabular}{|c | p{.39\linewidth} | c | c |}
 \hline
 $\hspace{0.25 cm}p \hspace{0.25 cm} $ & \centering{Base $\ep_F$ Expansion of $1/p$} & $\ell_{F,p}$ & $h_{F_(\sqrt{-p})}$ Calculation \\  
 \hline\hline
 $7$ & $0.01\hspace{+0.07 cm}\overline{32202230}$ & $8$ & $\frac{1}{4}\bigg(-\frac{220}{7} + \frac{228}{7} - \frac{188}{7} + \frac{212}{7} - \frac{180}{7} + \frac{204}{7}\bigg)=2$ \\ 
 \hline
 $19$ & $0.002\hspace{+0.07 cm}\overline{22231}$ & $5$ & $\frac{1}{4}\bigg(\frac{396}{19}+\frac{400}{19}-\frac{360}{19}+\ldots+\frac{332}{19}+\frac{328}{19}\bigg)=2$\\ 
 \hline
 $31$ & $0.001\hspace{+0.07 cm}\overline{2132023120322221002303200122}$\newline$\overline{2230}$ & $32$ & $\frac{1}{4}\bigg(\frac{3876}{31}-\frac{3788}{31}-\frac{3764}{31}-\ldots -\frac{3444}{31}+\frac{3420}{31}\bigg)=6$\\ 
 \hline
 $43$ & $0.001\hspace{+0.07 cm}\overline{02311222230}$ & $11$ & $\frac{1}{4}\bigg(\frac{1856}{43}+\frac{1848}{43}+\frac{1940}{43}+\ldots+\frac{1656}{43}+\frac{1624}{43}\bigg)=6$\\ 
 \hline
 $67$ & $0.0002\hspace{+0.07 cm}\overline{3110011313222221320122102312}$\newline$\overline{222231}$& $34$ & $\frac{1}{4}\bigg(\frac{8908}{67} +\frac{9228}{67}-\frac{8604}{67}+\ldots+\frac{7492}{67}+\frac{7860}{67}\bigg)=6$\\ 
 \hline
 $79$ & $0.0002\hspace{+0.07 cm}\overline{122101031011213031211013010}$\newline$\overline{122113222222010012113021100303001}$\newline$\overline{12031121001022222231}$& $80$ & $\frac{1}{4}\bigg(\frac{22740}{79}
-\frac{22364}{79}+\ldots-\frac{22372}{79}+\frac{22500}{79}\bigg)=30$\\ 
 \hline
\end{tabular}
\caption{Theorem \ref{thm 1.2} for primes $p<100.$}
\label{table_thm1.2}
\end{table}
% \begin{thebibliography}{99}
% \bibitem{wei_lun_stark_units}
% Adrian Barquero-Sanchez, Riad Masri,  and  Wei-Lun Tsai. Stark units and special Gamma values. \emph{Res. Number Theory,} $7, 06$ $2021.$
% \bibitem{sage}
% William A. Stein. dnkfks,\texttt{https://nwjnkwe}, 2018.
% \end{thebibliography}
\bibliography{bibliography}{}
\bibliographystyle{plain}
\end{document}